\documentclass[11pt,reqno]{amsart}

\usepackage{graphicx} % Required for inserting images
\usepackage{amsmath}
\usepackage{mathtools}
\usepackage{mathrsfs}
\usepackage{amssymb}
\usepackage{appendix}
\usepackage{mathtools}
\usepackage{xurl}
\usepackage{hyperref}

\usepackage{microtype}

\usepackage{blindtext}
\usepackage[dvipsnames]{xcolor}
\usepackage[normalem]{ulem}
\usepackage{soul}
\usepackage[backend=biber,style=alphabetic]{biblatex}

\usepackage{bm}
\usepackage{hyperref}

\newtheorem{theorem}{Theorem}[section]
\newtheorem{corollary}[theorem]{Corollary}
\newtheorem{lemma}[theorem]{Lemma}

\newtheorem{proposition}[theorem]{Proposition}
\newtheorem{remark}[theorem]{Remark}

\newtheorem{claim}[theorem]{Claim}

\newcommand{\bb}[1]{\mathbb{#1}}

\newcommand{\norm}[1]{\left|#1\right|}

\newcommand{\ra}{\rightarrow}
\newcommand{\parent}[1]{\left( #1 \right)}

\newcommand{\bR}{\bb{R}}

\newcommand{\eps}{\epsilon}

\newcommand{\lda}{\lambda}

\newcommand{\ric}{\text{Ric}}

\newcommand{\Lda}{\Lambda}

\newcommand{\rom}[1]{%
  \textup{\uppercase\expandafter{\romannumeral#1}}%
}

\newcommand{\fl}{\text{flat}}
\newcommand{\bD}{\mathbb{D}}
\newcommand{\old}{\overline{\mathbb{D}}}
\newcommand{\hpi}{\frac{\pi}{2}}

\numberwithin{equation}{section}
\AtBeginBibliography{%
  \setlength{\emergencystretch}{5em}%
}
\title[Uniqueness of 2D Ancient Ricci Flows with Boundary]{Uniqueness of Positively Curved Ancient Ricci Flows on Surfaces with Boundary}
\author{Kyeongho Bang}
\address{Department of Mathematical Sciences, Korea Advanced Institute of Science and Technology, 34141 Daejeon, Korea}
\email{khbang01@kaist.ac.kr}
\author{Eric Chen}
\address{Department of Mathematics, University of Illinois Urbana-Champaign, Urbana, Illinois 61801,  USA}
\email{ecchen@illinois.edu}
\author{Wenkui Du}
\address{School of Mathematics, Hunan University, Changsha, 410082, Hunan Province, China}
\email{duwenkui@hnu.edu.cn}

\date{}

\begin{document}

\begin{abstract}
    We establish the existence and uniqueness modulo time-independent diffeomorphisms of the positively curved ancient Ricci flow $(M^2, \partial M^2, g(t))$ on a two-dimensional surface with boundary, assuming uniformly bounded diameter and constant positive boundary geodesic curvature. In particular, this ancient Ricci flow is rotationally symmetric, its backward limit is the flat disk, and its forward limit is a half-spherical singularity.  To our knowledge, this result is the first instance of a classification result for ancient Ricci flows with boundary.
\end{abstract}

\maketitle

\tableofcontents

\section{Introduction}

Ancient flows are solutions of geometric flows defined from time minus infinity. They arise naturally as models for singularity formation and play a fundamental role in understanding the global structure of geometric flows. For two-dimensional Ricci flow, the classification theorem of Daskalopoulos, Hamilton, and Šešum in \cite{DHS12} shows that compact ancient solutions with positive curvature are remarkably rigid: up to the natural symmetries of the flow, they are given by the shrinking round sphere or the King–Rosenau solution.  In higher dimensional Ricci flow, the study and classification of ancient $\kappa$-solutions constitute a central component of the analysis of singularities and canonical neighborhoods (see for instance \cite{haslhofer2024kappa}, \cite{ABDS}, \cite{BDS2021uniqueness}, \cite{Bre-3d-noncpt}, \cite{lai2024family}, \cite{BDNS-compact}, \cite{BN-noncompact}, \cite{zhao20234}, \cite{ZZ-notype2}, \cite{ma2023unique}, \cite{hebbar2026asymptotic},\cite{hebbar2026kappa}, \cite{daskalopoulos2026unique}). On the other hand, for mean curvature flow, which shares many qualitative features with Ricci flow, the parallel classification program of ancient boundaryless mean curvature flow has been growing rapidly (see \cite{choi2026classification}, \cite{ChoiMantoulidis}, \cite{BC1}, \cite{ADS1}, \cite{ADS2}, \cite{CHH}, \cite{CHHW}, \cite{choi2025revisiting}, \cite{zhu2022so}, \cite{CHH_translator}, \cite{choi2025classification}, \cite{CH-classification-r4}, \cite{CDZ25}, \cite{CDZ26}, \cite{bamler2025pde}, \cite{bamler2025classification}). 

Recently, Bourni, Langford, Burns, Catron and Franz have pioneered the program of classifying ancient mean curvature flows with free boundary in \cite{BL23}, \cite{BL25},  \cite{bourni2025classification}, \cite{BF26}. However, the classification of ancient Ricci flows subject to geometric boundary conditions appears to be largely undeveloped even in dimension two. Indeed, the study of Ricci flow on manifolds with boundary  presents essential difficulties that have no direct analogue in the closed setting. On a manifold with boundary, one must impose boundary conditions that are simultaneously geometrically natural, compatible with the gauge fixing, and sufficient to produce a well-posed parabolic boundary value problem. Typical geometric conditions prescribe the conformal class of the induced boundary metric together with the mean curvature, impose umbilicity, or prescribe the geodesic curvature in dimension two (see \cite{Shen96}, \cite{Pul13}, \cite{brendle2002curvature}, \cite{cortissoz2019ricci}, \cite{Gia16a}, \cite{Gia16b}, \cite{Chow22}, \cite{li2025normalized} for details).

In this paper, we study the two-dimensional  nonnegatively curved ancient Ricci flow $(M^2, \partial M^2, g(t))$ on a compact surface with boundary, where the ancient Ricci flow has uniformly bounded diameter\footnote{By the  decreasing property of diameter along nonnegatively curved Ricci flow, uniformly bounded
diameter condition is equivalent to $\lim_{t\to -\infty}\textrm{diam}(M^2,  g(t))\leq D<+\infty$.} and boundary subject to a prescribed positive constant geodesic curvature boundary condition:
\begin{align}\label{2DRF}
\begin{cases}
\partial_t g = -2\textrm{Ric}(g) = -R_g\,g,
& \forall (x,t)\in M\times(-\infty,T),\\[0.3em]
g(0)=g_0,
& x\in M,\\[0.3em]
k_{\partial M}(t)\equiv k_{\partial M}(0)=k>0,
& \forall (x,t)\in \partial M\times(-\infty,T),\\[0.3em]
R_g\ge 0, \textrm{diam}(M^2,  g(t))\leq D<+\infty,
& \forall (x,t)\in M\times(-\infty,T).
\end{cases}
\end{align}
where $R_{g}$ is the scalar curvature and $k>0$ is a fixed constant.

By the Gauss-Bonnet theorem, the positive curvature condition $R_g\geq 0$ combined with the positive boundary geodesic curvature imply that the surface $M^2$ must be a closed two-dimensional disk $\old$. 
And by \cite[Theorem 1.1]{CoMu19}, an ancient solution Ricci flow $(\old, \partial \bD, g(t))$ of \eqref{2DRF} on a surface with boundary of positive Gaussian curvature and positive constant geodesic curvature boundary forms a half-spherical singularity in finite time, which allows us to reduce the flow equation \eqref{2DRF} into a boundary value problem of the conformal factor with respect to a background metric and coordinates given by either the standard round hemisphere or the standard flat metric. 
 
\subsection{Main results}

We state below our the main results on the existence and uniqueness of ancient solutions to the Ricci flow with boundary \eqref{2DRF}.

\begin{theorem}\label{main existence uniqueness theorem}
    Modulo time translation and time independent diffeomorphism, there is a unique positively curved ancient solution to the {Ricci flow with boundary} \eqref{2DRF} having uniformly bounded diameter. In particular, this unique ancient solution is a rotationally symmetric ancient Ricci flow with backward limit a radius $\frac{1}{k}$-flat disk and forward limit a half-spherical singularity.
\end{theorem}
By a standard maximum principle argument, it is not hard to see that the ancient solution $(\old, \partial \bD, g(t))$ of \eqref{2DRF}  either is isometric to the static flat disk with radius $k^{-1}$ or has positive Gaussian curvature everywhere. Hence, we have the following immediate consequence of Theorem \ref{main existence uniqueness theorem}.
\begin{corollary}
    Any two-dimensional nonnegatively curved ancient Ricci flow $(M^2, \partial M^2, g(t))$ on a surface with boundary satisfying \eqref{2DRF} and having uniformly bounded diameter must be, up to time translation and time independent diffeomorphism, either the flat disk with radius $\frac{1}{k}$ or the unique rotationally symmetric positively curved ancient Ricci flow with constant boundary geodesic curvature $k$  constructed in Theorem \ref{main existence uniqueness theorem}.
\end{corollary}

Theorem \ref{main existence uniqueness theorem} is a direct consequence of Proposition \ref{existence} (existence) and Proposition \ref{uniqueness of 1.1} (uniqueness) in this paper.

Our proof strategy is inspired in part by the strategy in the works of Bourni--Langford \cite{BL23, BL25} about the classification of convex ancient free-boundary mean curvature flows in the ball. However, for ancient Ricci flows with boundary, there several difficulties to be overcome which have no direct counterpart in the free-boundary mean ancient curvature flow setting. 

First, there is no ambient geometry in the Ricci-flow setting, and hence no natural height function or other extrinsic geometric quantities; nor is there any explicit family of geometric barriers analogous to those available for free-boundary mean curvature flow. We must instead establish a boundary-compatible nonlinear comparison principle, intrinsically construct suitable barriers and approximate old solutions with desired properties, and also identify an appropriate intrinsic analogue of the height function which is essential for establishing both existence and uniqueness. In addition, when establishing rotational symmetry, the Alexandrov reflection argument used in the extrinsic setting does not directly apply to our setting.

Second, to establish uniqueness in our intrinsic setting, there is neither a nested family of ambient regions nor a direct geometric mechanism selecting the backward limiting metric, and we have to rule out curvature concentration or degeneration near the boundary. 

Third, conformal gauge freedom presents a further difficulty in our setting. A priori, convergence to the flat disk does not select a preferred conformal parametrization, and without accounting for this equivalence there is in fact a two-dimensional family of stationary solutions which may arise as backward limits. These must be interpreted appropriately in terms of the spectral information of the linearized equation in order to show the uniqueness modulo time-independent diffeomorphisms of our backward limit.

\subsection{Outline of the paper}

In Section \ref{setup}, we set up the basic notations and evolution equations associated with two-dimensional Ricci flow with boundary. In Section  \ref{barrier3}, we construct some special barriers and discuss the necessary comparison principles for various geometric quantities. In Section  \ref{robinbcspectrum}, we recall some spectral properties of the linear critical Robin heat equation. In Section \ref{construction of arf}, we construct suitable old solutions and use them to establish the existence of the ancient Ricci flow (Theorem \ref{main existence uniqueness theorem}). In  Section  \ref{flatbackwardlimit}, we show that the backward limits of the ancient solutions are the standard flat disk up to time independent diffeomorphisms. In Section  \ref{symmetrysection}, we show the rotational symmetry of ancient solutions converging to a flat disk. Finally, in Section \ref{uniquenesssection}, we prove the uniqueness part of Theorem \ref{main existence uniqueness theorem} for our constructed ancient solution using its asymptotics.

\subsection{Acknowledgments} The first author was supported by the National Research Foundation of Korea (NRF) grant funded by the Korean government (MSIT) RS-2024-00346651 and the Korea Advanced Institute of Science \& Technology. The second author was supported by a Simons Travel Grant and the University of Illinois Urbana-Champaign.  The third author was supported by Hunan University.  

\subsection{AI use disclosure}
Apart from literature searches and minor proofreading tasks, the authors acknowledge the use of AI models exclusively in the following: to find the test function $S(\psi, t)$  \eqref{Stestf}  in the proof  of Lemma \ref{R uperbound estimates}, to speed up the computation of the evolution equation of $E(\psi, t)$ in  Lemma \ref{evolution of E}, and to derive the decay estimates in Lemma \ref{lem_decay_estimate}. All writing in this paper was done by the authors, who take responsibility for its correctness.

\section{Preliminaries}\label{setup}
In this section, we discuss some basic features of the scalar evolution equations associated with a two-dimensional Ricci flow with boundary, and further properties in the special case of rotational symmetry. We will need these later for our constructions and quantitative analyses. 

Let $g_{0}$ denote the standard unit half-spherical metric on the closed two-dimensional disk $\old$, with totally geodesic boundary $g_{0}|_{\partial \bD}$. Denote by $\nabla^0$ the induced Levi-Civita connection, and by $\Delta_0$ the Laplacian on the hemisphere $(\overline{\mathbb{S}^2_+},g_0)$ induced from the usual Laplacian on 2-sphere $\Delta_{\mathbb{S}^2}$. 
Since the two-dimensional Ricci flow preserves the conformal class, we can parametrize the Ricci flow by the limiting half-sphere at the singular time $T$; that is, we can write
$$g(\psi, \theta,t) = u(\psi, \theta, t)g_{0},$$
where $\psi\in [0, \frac{\pi}{2}]$ measures the colatitude and $\theta\in [0, 2\pi]$ measures the longitude.

We can then rewrite (\ref{2DRF}) in terms of conformal factor $u(\cdot,t)$ as follows:
\begin{align}\label{2Du}
    \begin{cases}
        u_t = -R(u)\cdot u = \Delta_0\log u - 2< 0\qquad &\text{in }\old\times (-\infty,T),\\
        \frac{\partial u}{\partial \nu_0} = 2ku^{3/2}\qquad  &\text{on }\partial \bD \times (-\infty,T). 
    \end{cases}
\end{align}
where $\nu_0$ denotes the unit outward normal vector at the boundary with respect to $g_{0}$. This is a nonlinear Robin boundary problem for the conformal factor $u := u(p,t)$. 
A stationary solution to the above equation (and therefore ancient) is the radius $\frac{1}{k}$-flat disk metric $g_{\text{flat},k}(t) \equiv g_{\text{flat}}$, which in the coordinates $(\psi,\theta)$ can be written as
\begin{equation}\label{eq_flat}
g_{\text{flat},k}(\psi,\theta)  = u_{\text{flat},k}(\psi)g_0 :=  \frac{1}{k^2(1+\cos\psi)^2}(d\psi^2 + \sin^2\psi d\theta^2).
\end{equation}
For convenience of notation we will often suppress the constant $k$ in writing $g_{\fl}$ and $u_{\fl}$.

For the scalar curvature $R = R(p,t)$, the time independent diffeomorphism formulas for the time-dependent Laplacian and gradient 
\begin{equation}\label{time dependent lg}
    \Delta = \Delta_{u(t)g_0}=u(t)^{-1}\Delta_0,\qquad \nabla = \nabla_{u(t)g_0}= u(t)^{-1}\nabla^0,
\end{equation}
together with \eqref{2Du} yield the following equations:
    \begin{equation}
        \begin{cases}
                R(u) = \frac{1}{u}(2-\Delta_{0} \log u)>0\\
                \partial_tR = \Delta_{g(t)} R + R^2 = u(t)^{-1}\Delta_0R + R^2.\\
                \frac{\partial R}{\partial \nu_{g(t)}} = \frac{1}{\sqrt{u(t)}}\frac{\partial R}{\partial\nu_0} = kR.
            \end{cases}
    \end{equation}
In our later construction of the nonstatic rotationally symmetric ancient solution to (\ref{2DRF}), we will consider the  rotationally symmetric conformal factor function $u = u(\psi)$ in $(\mathbb{S}^2_+,g_0,\nabla^0)$, for which we have the following explicit formulas:
$$|\nabla^0 u|_{g_0}^2 = \parent{\frac{\partial u}{\partial \psi}}^2,\qquad \Delta_0 u = u_{\psi\psi} + \cot\psi\;u_\psi.$$
Therefore, a rotationally symmetric solution $u(\psi,\theta,t) = u(\psi,t)$ to (\ref{2Du})  solves the following boundary value problem:
\begin{align}\label{2Dsym}
        \begin{cases}
            u_t = \Delta_{0}\log u - 2 = (\log u)_{\psi\psi} + \cot\psi \;(\log u)_{\psi} - 2\; &\text{for all } \psi\in [0,\pi/2),\\
            \frac{\partial u}{\partial \nu_0}  = u_\psi =2ku^{3/2}\; &\text{for }\psi = \pi/2. 
        \end{cases}
    \end{align}

\section{Comparison principles and barriers}\label{barrier3}
In this section, we establish comparison principles for various geometric quantities and construct special barriers.
\subsection{Locating extrema}
We first present two comparison principles between a subsolution and a solution to the rotationally symmetric Ricci flow with boundary equation (\ref{2Dsym}). These will show that two important properties of our initial data---the signs of the gradients $u_\psi,\; R_\psi\geq 0$---are preserved along the flow. Using these comparison principles, we can keep track of the precise location of points achieving the extremal values of $u(t)$ and $R(t)$.

\begin{lemma}[Location of conformal factor extremum]\label{confloc}
    Let $g(\psi,\theta,t) = u(\psi,t)g_{0}$, $0\leq t< T$ be a radially symmetric {Ricci flow with boundary} (i.e. not dependent on $\theta$) with the initial data $u(\psi,0) = u_0(\psi)$ satisfying $\partial_\psi u_0 \geq 0$. Then for $0\leq t<T$ and all $\psi\in[0,\frac{\pi}{2}]$ we have
    $$\partial_\psi u(\psi,t)\geq0.$$
    Consequently, the conformal factor $u(\psi,t)$ always achieves its maximum value on the boundary $\{\psi = \pi/2\}$ and its minimum value at the pole $\{\psi=0\}$, whenever the initial data exhibits the same behavior. Furthermore, if the inequality is strict away from the pole, it remains so throughout the evolution.
\end{lemma}
\begin{proof}
    Define $w(\psi,t) = \log u$. Then $w$ satisfies the equation
    $$\begin{cases}
        w_t = e^{-w}(w_{\psi\psi} + \cot\psi \;w_\psi - 2),\\
        \partial_\nu w = 2ke^{w/2}>0.
    \end{cases}$$
    Differentiating in $\psi$, we see that $v := w_\psi = (\log u)_\psi$ satisfies the Dirichlet boundary value problem
    $$\begin{cases}
        v_t = e^{-w}\parent{v_{\psi\psi} + (-v + \cot\psi)v_\psi + (2-\cot\psi\; v - \csc^2\psi)v },\\
        v(\pi/2,t) = 2ke^{w/2}>0.
    \end{cases}$$
    Notice that the operator on $v$ is singular at $\psi =0$. To remedy this issue, we introduce the new function
    $$q(\psi,t) := \frac{v(\psi,t)}{\sin\psi} = \frac{u_\psi}{u\sin\psi},$$
    which is well-defined and smooth due to the rotational symmetry of $g(t)$.
    Then the function $q$ satisfies the Dirichlet boundary problem
    $$\begin{cases}
        q_t = e^{-w}\parent{q_{\psi\psi} + 3\cot\psi~q_\psi - \sin\psi~qq_\psi - 2\cos\psi~q^2},\\
        q(\pi/2,t) = v(\pi/2,t) >0.
    \end{cases}$$
    If we now view $q$ as a rotationally symmetric function on $\overline{\mathbb{S}^4_+}$, then we can observe that $q_{\psi\psi} + 3\cot\psi~q_\psi = \Delta_{\overline{\mathbb{S}^4_+}}q$, where $\Delta_{\overline{\mathbb{S}^4_+}}$ denotes the Laplacian on the hemisphere $\overline{\mathbb{S}^4_+}$. Therefore, $q$ satisfies a uniformly parabolic equation with a Dirichlet boundary condition on $\overline{\mathbb{S}^4_+}$. Then we can apply the parabolic maximum principle to see that $q(\cdot,t) \geq0$ persists throughout the evolution, which yields $\partial_\psi u(\psi,t) \geq 0$ for all $0\leq t<T$ (at the pole, $u_\psi =0$ by smoothness). Last statement of the lemma follows from Sturm's theory \cite{Angenent1988}.
\end{proof}
Applying the arguments in the proof of {Lemma \ref{confloc}} almost verbatim, we also obtain the following comparison result:
\begin{lemma}[Location of curvature extremum]\label{curvloc}
         Let $g(\psi,\theta,t) = u(\psi,t)g_{S^2}$, $0\leq t< T$ be a radially symmetric {Ricci flow with boundary} (i.e. not dependent on $\theta$) with the initial data $R(\psi,0) = R_0(\psi)$ satisfying $\partial_\psi R_0 \geq 0$. Then for $0\leq t<T$ and all $\psi\in[0,\frac{\pi}{2}]$ we have
    $$\partial_\psi R(\psi,t)\geq0.$$
    Consequently, the scalar curvature $R(\psi,t)$ of $g(\psi,t)$ always achieves its maximumum value on the boundary $\{\psi = \pi/2\}$ and its minimum value at the pole $\{\psi=0\}$, whenever the initial data exhibits the same behavior. Furthermore, if the inequality is strict away from the pole, it remains so throughout the evolution.
    \end{lemma}
    \begin{proof}
    If we set $q := \frac{R_\psi}{\sin\psi}$, then $q$ satisfies the {boundary value problem}
    $$\begin{cases}
        q_t = u^{-1}\parent{q_{\psi\psi} + 3\cot\psi q_\psi - 2q} - 2\frac{u_\psi}{u^2}(q_\psi + 2q\cot\psi) + 2Rq,\\
        q(\pi/2,t) = R_\psi(\pi/2,t) >0.
    \end{cases}$$
    As long as $R$ exists and stays bounded, we can use the maximum principle as in the proof of Lemma \ref{confloc} to prove that $\min q$ remains non-negative, which implies that $R_\psi$ stays nonnegative. The last statement of the lemma follows from Sturm's theory \cite{Angenent1988}.
    \end{proof}
The above scalar curvature comparison lemma immediately yields the following interesting geometric consequence:
\begin{lemma}[Monotonicity of the isoperimetric ratio]\label{isop}
        Given a metric $(\old,\partial\bD ,g)$, define the isoperimetric ratio
    $$I(g) := \frac{L(\partial \bD,g)^2}{A(\old,g)}.$$
    Let $\{g(t) = u(\psi,t)g_0\}_{t\in [t_0,t_1]}$ be a positively curved rotationally symmetric solution to (\ref{2DRF}) with the initial condition $\partial_\psi u(\cdot,t_0),\; \partial_\psi R(\cdot,t_0) \geq 0$. Then we have
    $$\frac{d}{dt}I(g(t)) \leq 0,\; \forall t\in (t_0,t_1).$$
    \end{lemma}
    \begin{proof}
    We first compute the time derivatives of $L(t) = L(\partial \bD,g(t))$ and $A(t) = A(\old,g(t))$:
    \begin{align*}
        L'(t) &= \partial_t \int_{\partial \bD} ds(t) = \int_{\partial \bD} -\frac{R}{2} ds(t),\\
        A'(t) &= \partial_t \int_{\old} dg(t) = \int_{\old} -R\;dg(t).
    \end{align*}
    Therefore $I(g(t))$ satisfies
    \begin{align}
        \frac{dI}{dt} &= \frac{2L(t)L'(t)}{A(t)} - \frac{L(t)^2 A'(t)}{A^2(t)}\\
        &= I(t)\parent{\frac{1}{A(t)}\int_{\old}R\;dg - \frac{1}{L(t)}\int_{\partial \bD}R \;dg_{\partial \bD}}\notag\\
        &=: I(t)\parent{\bar{R}_{\old(t)} - \bar{R}_{\partial \bD}(t)}.
    \end{align}
    So by {Lemma \ref{curvloc}}, we have $I'(t)\leq 0$.
    \end{proof}

\subsection{Spherical cap barriers and a comparison principle}

We now find reasonable barriers to control the existence time for our to-be-constructed solutions defined on $[-T_i,0)$, $T_i\ra \infty$.

A spherical cap in the sphere of radius $r>0$ with boundary curvature $k>0$ can be written as a metric over the hemisphere
    $$g(k,r)=\frac{r^2 \left(d\psi^2+\sin^2\psi\,d\theta^2\right)}{\big(\sqrt{1+r^2k^2}+rk\cos\psi\big)^2},\;
    \theta\in[0,2\pi],\;\psi\in\left[0,\frac{\pi}{2}\right].$$
    We see that as $r\ra \infty$, the metric $g(k,r)$ converges to the metric of a flat disk with radius $1/k$; as $r\ra 0^+$, the metric vanishes.\\
    If we set $r=r(t)$ for $t<0$, we have
    \begin{align*}
        \partial_t g(k,r(t)) &=\frac{2 r'(t)}{\,r(t)\sqrt{1+k^2r(t)^2}\,\bigl(\sqrt{1+k^2r(t)^2}+k\,r(t)\cos\psi\bigr)}\,g(t),\\
        \ric(g(k,r(t))) &= r(t)^{-2}g(t).
    \end{align*}
    To find super/subsolutions to the Ricci flow equation with constant boundary curvature, it suffices to find $r_1,r_2:(-\infty,0)\ra \bR_{+}$ satisfying
    \begin{equation}\label{capineq}
    \frac{r_1'}{1+k^2r_1^2} \geq -\frac{1}{r_1}\quad\quad \text{and} \quad\quad \frac{r_2'}{2(1+k^2r_2^2)}\leq -\frac{1}{r_2}
    \end{equation}
    because then
    $g(k,r_1(t))$ (supersolution) satisfies $$\partial_tg(k,r_1(t))\geq -2\ric(k,r_1(t))$$ and $g(k,r_2(t))$ (subsolution) satisfies $$\partial_tg(k,r_2(t))\leq -2\ric (k,r_2(t)),$$ both with the boundary condition of constant curvature $k>0$. A direct computation yields that $$r_1(t) = \frac{\sqrt{e^{-2k^2t} - 1}}{k}$$ and $$r_2(t) = \frac{\sqrt{e^{-4k^2t}-1}}{k}$$ easily satisfy (\ref{capineq}) with the vanishing condition $r_1(0) = r_2(0) = 0$.\\
    Indeed, setting $$U_i(t) := \frac{r_i^2}{\left(\sqrt{1+r_i^2k^2}+r_ik\cos\psi\right)^2}, \quad i=1,2,$$
    we see that $U_i(t)$'s are super/subsolutions to equation (\ref{2Du}) defined for all $t\in(-\infty,0)$ with matching boundary conditions at $\psi=\pi/2$. From now, we shall refer to these super/subsolutions as \textit{spherical cap barriers}.
    Applying the maximum principle to the conformal factors, we obtain the following nonlinear comparison principle:
    \begin{lemma}[Conformal factor comparison]\label{maxtime}
        Suppose we are given a solution to the (not necessarily rotationally symmetric) Ricci flow with boundary $u(\psi,\theta,t)$ starting from initial data $u_0 = u_0(\psi,\theta)$. If for a time-slice $U_1(-T)$ (resp. $U_2(-T)$) of the spherical cap supersolution $U_1$ (resp. subsolution $U_2$) one has $u_0\leq U_1(-T)$  (resp. $\geq$), then $u(t) \leq U_1(t-T)$ (resp. $\geq$) for all $t>0$ for which the flows are defined.
    \end{lemma}
    \begin{proof}
        We prove the case in which one compares with a subsolution. For brevity, denote $v=U_2$. The conformal factor $u(t)$ satisfies the equation
    $$\begin{cases}
        u_t = \Delta_0\log u -2,\\
        \frac{u_{\nu_0}}{2u^{3/2}} = -\parent{\frac{1}{\sqrt{u}}}_{\nu_0}= k >0.
    \end{cases}$$
    So $u-v$ satisfies the equation
    $$\begin{cases}
        (u-v)_t \geq \Delta_{0}(\log u - \log v)\\
        \frac{\partial_{\nu_0}u}{2u^{3/2}} = \frac{\partial_{\nu_0} v}{2v^{3/2}} \iff \parent{\frac{1}{\sqrt{u}} - \frac{1}{\sqrt{v}}}_{\nu_0} = 0.
    \end{cases}$$ 

    To use Hopf's boundary point lemma, we apply the maximum principle on the function $\eta := \frac{1}{\sqrt{v}} - \frac{1}{\sqrt{u}} =: b-a$. Via straightforward computation, we see that $\eta$ satisfies the equation
    $$\begin{cases}
        \eta_t \geq (b^2\Delta_0 b - b|\nabla^0 b|^2 + b^3)-(a^2\Delta_0 a - a|\nabla^0 a|^2 + a^3)  =: F(b) - F(a)\\
        \eta_{\nu_0} = 0.
    \end{cases}$$
    
    Notice that
    $$F(b) - F(a) = A\Delta_0 \eta + B\cdot \nabla^0 \eta + C\eta,$$
    where
    \begin{align*}
        \eta_s&:= (1-s)a + sb,\quad
        A := \int_0^1 \eta_s^2 ds >0,\\
        B&:= -2\int_0^1 \eta_s\nabla \eta_s\;ds,\quad
        C:= \int_0^1 2\eta_s\Delta \eta_s - |\nabla \eta_s|^2 + 3\eta_s^2\;ds.
    \end{align*}
    Note also that $A,B,C$ are bounded whenever either of $u$ or $v$ does not vanish, say in a compact interval $t\in I$. Choosing $\Lda > \sup_{t\in I} C$, define $$\tilde{\eta} := e^{-\Lda t}\eta.$$
    Then $\tilde{\eta}$ satisfies
    $$\begin{cases}
        \tilde{\eta}_t \geq A\Delta \tilde{\eta} + B\cdot \nabla \tilde{\eta} + (C-\Lda)\tilde{\eta},\\
        \tilde{\eta}_{\nu_0} = 0.
    \end{cases}$$
    By Hopf's boundary point lemma, $\tilde{\eta}$ cannot attain a strict boundary extremum and the minimum $\min_{t\in I} \tilde{\eta}(\cdot,t)$ has to be positive as long as $U_2(t-T)$ survives, given the initial data comparison. To see this, suppose that on the contrary say at $(\psi_0,t_0)\in [0,\pi/2)\times I$ we have $\min_{[0,\pi/2]\times I} \tilde{\eta} = \tilde{\eta}(\psi_0,t_0)<0$. Then the above inequality on $\tilde{\eta}$ yields
    \begin{align*}
        0\geq \tilde{\eta}_t(\psi_0,t_0) &\geq A\Delta \tilde{\eta}(\psi_0,t_0) + B\cdot \nabla \tilde{\eta}(\psi_0,t_0) + (C-\Lda)\tilde{\eta}(\psi_0,t_0)\\ &\geq (C-\Lda)\tilde{\eta}(\psi_0,t_0) >0,
    \end{align*}
    which is false.
    
    So we know that $u-v$ cannot reach a negative minimum as long as one of the two functions survives, and consequently $u$ does not vanish as long as $v$ survives.
    \end{proof}
We also record the following scalar curvature comparison principle without proof. It directly follows from the standard elliptic maximum principle applied to a pair of functions that satisfy the same boundary derivative condition.
\begin{lemma}[Scalar curvature comparison]\label{scalarcomp}
    Let $g_1  = u_1g_0$ and $g_2 = u_2g_0$ be two conformally equivalent metrics on the closed disk $\old$ with the same boundary curvature. Suppose that at some point $p\in \old$ (in the interior or boundary), $u_1(p) = u_2(p)$ and $u_1\geq u_2$ in a neighborhood of $p$. Then we have
    $$R_{g_1}(p) \leq R_{g_2}(p).$$
\end{lemma}

\section{The critical Robin heat equation and spectral analysis}\label{robinbcspectrum}
In this section, we recall some spectral properties of the linear critical Robin heat equation. This arises from a linearization of \eqref{2Du} about its flat stationary solution and will play an important role in both our construction of an ancient solution and the proof of its uniqueness. 
\subsection{A linearized operator and its eigenvalue problem.} Let $u(t)$ denote an ancient solution of the {Ricci flow with boundary} \eqref{2Du}. Then $w := 1/\sqrt{u}$ satisfies the equation
$$\begin{cases}
    w_t = w^3(\Delta_0 \log w +1).\\
    \partial_{\nu_0}w \equiv -k.
\end{cases}$$
Consequently the profile function $$v := \frac{1}{\sqrt{u}} - \frac{1}{\sqrt{u_\fl}}$$ satisfies the equation
\begin{equation}\label{eq_profile}
\begin{cases}
    v_t = w_t = \big(v+\frac{1}{\sqrt{u_\fl}}\big)^3\big(\Delta_0\log (v+\frac{1}{\sqrt{u_\fl}}) +1\big),\\
    \partial_{\nu_0}v \equiv 0,
\end{cases}
\end{equation}
where $u_{\fl}$ is as defined in \eqref{eq_flat}. To describe the asymptotics of the flow near $t=-\infty$, we study its linearization about the function 
\begin{equation}\label{eq_w}
w_\infty(\psi) := u_\fl^{-1/2}(\psi) =  k(1+\cos\psi)
\end{equation}
which we prescribe as the backward limit of $1/\sqrt{u(t)}$: 
$$\begin{cases}
   \Phi_t =  w_\infty^3\Delta_0 \parent{\frac{\Phi}{w_\infty}},\\
    \partial_{\nu_0}\Phi = 0.
\end{cases}$$
Separating variables with $\Phi(t,\cdot) = T(t)\phi(\cdot)$, we are led to consider the eigenfunctions to the following {boundary value problem}:
\begin{equation}\label{eigen{boundary value problem}}
\begin{cases}
    w_\infty^3\Delta_0\parent{\frac{\phi}{w_\infty}} = \lda_k \phi,\\
    \partial_{\nu_0}\phi = 0.
\end{cases}
\end{equation}
After a change of coordinates, the above problem can be written as a linear {boundary value problem} over a flat disk in $\bR^2$. Set
$$f := \frac{\phi}{w_\infty},\qquad r = \tan \frac{\psi}{2}.$$
Then the function $w_\infty$ can be written as
$$w_\infty = w_\infty(r) = \frac{2k}{1+r^2};$$
and the {boundary value problem} (\ref{eigen{boundary value problem}}) can be written as
\begin{align}\label{critrobin}
    \begin{cases}
        \Delta_{\bR^2}f= k^{-2}\lda_k f = :\mu f \qquad&\text{ in } \old,\\
        \partial_rf = f \qquad&\text{ on } \partial \bD,
    \end{cases}
\end{align}
where $\old\subset \bR^2$ denotes as before the Euclidean disk of radius $1$.

\subsection{A unique positive eigenfunction and its properties}\label{subsec_eigenfunction}
 In this subsection, we recall the  results of 
  the eigenvalue problem (\ref{critrobin}) discussed in the works of Bourni--Langford (\cite{BL23, BL25}). First recall that by standard methods, the eigenvalues of \eqref{critrobin} satisfy
  \[
  \mu_0>\mu_1\geq\mu_2\geq\mu_3\geq\cdots
  \]
  with $\mu_k\rightarrow -\infty$ as $k\rightarrow \infty$ (cf. \cite[\S6.5.1]{Evans}).
 \begin{lemma}[cf. {\cite[Lemma 2.1]{BL25}}]\label{lem_eigenfunction}
     The positive eigenspace of (\ref{critrobin}) is one-dimensional. Moreover, the positive eigenfunction $f_0 = f_0(r)$ is radial and can be written as a series
     $$f_0(r) = a_0\sum_{j=0}^\infty \frac{\mu_0^j}{4^j(j!)^2} r^{2j}$$
     where $a_0\in \bR$ is a constant and $\mu_0 >1$ is the positive eigenvalue corresponding to $f_0$ defined by
     \begin{equation}\label{mu_0}
    \sum_{j=1}^\infty \frac{(2j-1) \mu_0^j}{4^j(j!)^2} = 1.
    \end{equation}
 \end{lemma}
 \begin{proof}
 The proof is the same as in \cite[Lemma 2.1]{BL25}
\end{proof}
Below we record two useful analytic properties of $f_0(r)$.
\begin{lemma}\label{f_0'(r)}
    The eigenfunction $f_0(r)$ satisfies the inequalities
    \begin{align}
    \mu_0f_0^2(r) &\geq (f_0'(r))^2,\\
    \mu_0f_0 &\geq f_0''.
    \end{align}
\end{lemma}
\begin{proof}
The first inequality follows from
$$\frac{d}{dr}(\mu_0f_0^2 - (f_0')^2) = 2f_0'(\mu_0f_0-f_0'') = 2\frac{(f_0')^2}{r} \geq 0.$$
The second inequality follows from the radial Laplacian formula
$$\Delta_\bD f_0 = f_0'' + \frac{1}{r}f_0' = \mu_0f_0$$
and $f_0'\geq 0$.
\end{proof}
\begin{proposition}[cf. {\cite[Corollary 2.3]{BL25}}]\label{rulingoutnull}
    The null eigenspace of \eqref{critrobin} consists of linear functions. Therefore, any non-negative ancient solution $f(t)$ to \eqref{critrobin} satisfying $f(t) = e^{o(t)}$ as $t\to -\infty$ must be
    $$f(t) = Ae^{\mu_0t}f_0 \text{  for some $A\geq 0$.}$$
    
\end{proposition}
\begin{proof}
The proof is the same as \cite[Corollary 2.3]{BL25}.
\end{proof}

\section{Construction of a rotationally symmetric ancient solution}\label{construction of arf}
In this section, we construct a sequence of rotationally symmetric \textit{old solutions}---solutions which survive for a long enough time but which are not ancient---to the {Ricci flow with boundary} equation (\ref{2DRF}). We will use the eigenfunction $f_0$ (fixing $a_0=1$) from Section \ref{subsec_eigenfunction} to construct a suitable sequence of initial data for our old solutions which last for a sufficiently long amount of time.
\subsection{Old solutions}
Define a sequence of metrics
$$g_i = U(\cdot,\tau_i)g_0 := \frac{d\psi^2+ \sin^2\psi \;d\theta^2}{w_\infty^2(1+\eps(\tau_i)f_0(\psi))^2}$$
where $w_\infty$ is as in \eqref{eq_w} and we define
$$\eps(\tau) := e^{k^2\mu_0\tau},\; \tau<0.$$
The metrics $(g_i)_i$ also have the same constant boundary curvature $k_{i} \equiv k$:
\begin{align*}
    -k_i &=\partial_{\nu_0} \frac{1}{\sqrt{U(\psi,\tau_i)}} \\
    &=\partial_\psi|_{\psi=\pi/2} \frac{1}{\sqrt{U(\psi,\tau_i)}} \\
    &= \partial_\psi|_{\psi=\pi/2} (w_\infty(1+\eps(\tau_i)f_0(\psi))\\
    &= \partial_\psi w_\infty\parent{\hpi}\parent{1+\eps(\tau_i)f_0\parent{\hpi}} + w_\infty\parent{\hpi}\eps(\tau_i)\partial_\psi f_0\parent{\hpi}\\
    &= -k\parent{1+\eps(\tau_i)f_0\parent{\hpi}} + k\eps(\tau_i)f_0\parent{\hpi} = -k.
\end{align*}
We can also compute the scalar curvature of $U(\cdot,\cdot)g_0$ given by the formula
\begin{align*}
    R(U(\psi,t)g_0) &= \frac{1}{U(\psi,t)}\cdot \parent{2  -\Delta_0\log U(\psi,t)}\\
    &= \frac{1}{U(\psi,t)}\cdot \parent{2 - \Delta_0 \log u_\fl + 2\Delta_0 \log (1+\eps(t)f_0(\psi)) }\\
    &= \frac{1}{U(\psi,t)}\cdot 2\Delta_0\log (1+\eps(t)f_0)\\
    &= {2k^2}{(1+\eps(t)f_0(r))^2}\Delta_{\bD}\log (1+\eps(t)f_0(r))\\
    &= {2k^2\eps(t)}\parent{\mu_0f_0(r)(1+\eps(t)f_0^2(r)) - \eps(t)(\partial_r f_0(r))^2}\\
    &= 2k^2\eps(t)\parent{\mu_0f_0(r) + \eps(t)(\mu_0f_0^2(r)-(\partial_rf_0)^2)} \geq 0,
\end{align*}
from which it follows
\begin{align*}
    \lim_{t\to -\infty}R(U(\psi, t)g_0)=0.
\end{align*}
Note that by {Lemma \ref{f_0'(r)}} we then also have $R(U(r,t))> 0$ for any $t$ and $r$. Thus we can define a sequence of positively curved, rotationally symmetric {Ricci flows with boundary} $g_i(t)$ flowing out of the initial data $g_i(0) = g_i$. 
Now we shall invoke the spherical cap barriers constructed in the previous section to prove that the flows $g_i(t)$ survive for long enough times.

\begin{lemma}
    Denote the maximal existence interval of the old solutions $g_i(t)$ by $[0,t_i)$. Then $t_i \ra \infty$.
\end{lemma}
\begin{proof}
    This follows from {Lemma \ref{maxtime}} comparing the spherical caps to the initial data $g_i$, $g_i(\tau_i)$'s uniformly converge to the flat disk as $\tau_i\ra -\infty$.
\end{proof}
After a shift in time, denote the maximum existence interval of $g_i(t)$ as $[T_i,0)$, $T_i\ra -\infty$, so that the initial data $g_i(T_i)$ is precisely $ g_i$ and the flow vanishes at time $t=0$. Denote
$$g_i(t) = u_i(\cdot,t)g_0,\qquad u_i(\cdot,T_i) = U(\cdot,\tau_i).$$
We present two other nice properties of the initial data $g_i = g_i(T_i)$ of the old solutions $(g_i(t))_{t\in [T_i,0)}$.
\begin{lemma}\label{olddataproperties}
    The initial data $(g_i = U(\cdot,\tau_i)g_0)_i$ satisfy the following properties:
    \begin{align}
        \partial_\psi U(\psi,\tau_i) &\geq 0,\\
        \partial_\psi R(g_i)&\geq 0.
    \end{align}
    The above inequalities are strict away from the pole $\psi=0$.
\end{lemma}
\begin{proof}
It suffices to prove the claimed statement for the function $U(\psi,t)$ and the metric $g(t) := U(\psi,t)g_0$.
\begin{align*}
    \partial_\psi U(t) &= -\frac{2w_\infty^{-3}\partial_\psi w_\infty}{(1+\eps(t)f_0)^2} - \frac{2\eps(t)\partial_\psi f_0}{w_\infty^2(1+\eps(t)f_0)^3}\\
    &= \frac{2}{w_\infty^2(1+\eps(t)f_0)^2}\cdot \parent{-\frac{\partial_\psi w_\infty}{w_\infty} - \frac{\eps(t)\partial_\psi f_0}{1+\eps(t)f_0}}.
\end{align*}
We compute the bounds on $\partial_\psi f_0$.
\begin{align*}
    \partial_\psi f_0 &= \partial_r f_0\cdot \frac{\partial r}{\partial \psi} = \partial_r f_0 \cdot \frac{1}{2}\sec^2\parent{\frac{\psi}{2}}\\
    &= \partial_r f_0 \cdot \frac{1+r^2}{2} = \frac{1+r^2}{2}\cdot \sum_{j=1}^\infty \parent{\frac{2j\mu_0^j}{4^j(j!)^2}r^{2j-1}} < \infty,
\end{align*}
\begin{align*}
    \frac{\partial_\psi w_\infty}{w_\infty} &= \frac{1+r^2}{2}\cdot \frac{\partial_r w_\infty}{w_\infty}
    = \frac{1+r^2}{2}\cdot \parent{-\frac{2r}{1+r^2}} = -r.
\end{align*}
Then we have
\begin{align*}
    \partial_\psi U(t) = \frac{2}{w_\infty^2(1+\eps(t)f_0)^2}\cdot\parent{r - \frac{\eps(t)(1+r^2)\sum_{j\geq 1}\frac{j\mu_0^j}{4^j(j!)^2}r^{2j-1}}{1+\eps(t)\sum_{j\geq 0} \frac{\mu_0^j}{4^j (j!)^2}r^{2j}}}.
\end{align*}
To ensure that $\partial_\psi U(\psi,t) \geq 0$, it suffices to show
$$\parent{1- \eps(t)\cdot \frac{(1+r^2)\sum_{j\geq 1}\frac{j\mu_0^j}{4^j(j!)^2}r^{2j-2}}{1+\eps(t)\sum_{j\geq 0} \frac{\mu_0^j}{4^j (j!)^2}r^{2j}}} \geq 0$$
for sufficiently small $\eps(t) >0$. Since both sums in the denominator and numerator are uniformly bounded (from above and below) for $r\in [0,1]$, we can always find $T<0$ such that
$$\partial_\psi U(\psi,t) > 0,\; \forall t\leq T.$$
Using {Lemma \ref{f_0'(r)}}, we can also determine the sign of $R_\psi(U)$, or equivalently the sign of $R_r(U)$:
\begin{align*}
    \partial_r R(U(r,t)) &= 2k^2\eps(t)\parent{\mu_0f_0'(r) +\eps(t)(2\mu_0f_0f'_0(r) - \partial_r(\partial_rf_0)^2)}\\
    &= 2k^2\eps(t)(\mu_0f_0'(r) +2\eps(t)f_0'(r)(\mu_0f_0(r) - f_0''(r)))\\
    &\geq 2k^2\eps(t)\mu_0f_0'(r)
\end{align*}
as  $f_0'(r) \geq 0$ for all $r\in [0,1]$. The fact that the proposed inequalities are strict away from the pole $\psi=r=0$ trivially follows from the properties of $f_0$.
\end{proof}
In summary, for each old data $g_i$ we have
\begin{itemize}
    \item $R_\psi$, $U_\psi \geq 0$ (strict inequality away from the pole $\psi=r=0$),
    \item $R_i>0$ and $u_i$ bounded,
    \item $g_i = U(\cdot,\tau_i)g_0$ uniformly converges to the flat disk exponentially fast, as $i\ra \infty$.
\end{itemize}
Using the properties of the old solutions listed above, we prove some useful relations between their scalar curvatures and conformal factors. 
\begin{lemma}\label{scalandconformal}
Our constructed old solutions $u_i(\psi,t)$ satisfy
\begin{gather}
R_i(0,t)\leq \frac{2}{u_i(\hpi,t)} - 2k^2\leq R_i(\hpi,t),\\
{1 - k^2u_i\parent{\hpi,t}} \leq e^{2k^2t},\\
u_i(0,t)\leq \frac{u_i(\hpi,t)}{\parent{1+k\sqrt{u_i(\hpi,t})}^2}.
\end{gather}
\end{lemma}
\begin{proof}
    Let $u(\psi,t) = u_i(\psi,t)$ be one of our old solutions. We first prove the lower bound of $R(\hpi,t)$. Define
    $$U_k(\psi,t) := \frac{u(\hpi,t)}{(1+k\sqrt{u(\hpi,t})\cos\psi)^2}.$$
    $U_k(\psi,t)$ is a spherical cap of boundary curvature $k$ such that
    $$U_k\parent{\hpi,t} = u\parent{\hpi,t}.$$
    Suppose for some $t_0$, $U_k(t_0)$ lies below $u(t_0)$ locally near $\psi=\hpi$ (they intersect tangentially at $\psi = \hpi$). By {Lemma \ref{scalarcomp}}, this is equivalent to saying that the metric $U_k(t_0)$ has scalar curvature strictly greater than that of $u(t)$'s at the boundary $\psi = \hpi$. Now we construct a spherical cap metric $U_{k,k'}$ of same scalar curvature as $U_k(t)$ but with a different boundary curvature $k'>0$.
    Define
    $$U_{k,k'}(\psi,t) := \parent{\sqrt{\frac{1}{u\parent{\hpi,t}}-k^2+(k')^2} + k'\cos\psi}^{-2}.$$
    Notice that $U_{k,k} = U_k$. Observe that $U_{k',k}(\psi,t) >U_k(\psi,t)$ whenever $k'<k$ and vice versa. Hence, we can choose a small $k'<k$ so that $U_{k,k'}(t)$ makes first interior contact with $u(t)$ at some $\psi = \psi_0<\hpi$ from above. By {Lemma \ref{scalarcomp}}, we know that the metric $u(t)$ has the same scalar curvature at $\psi = \psi_0$ as its boundary point, which leads to a contradiction because each old solution maximizes its curvature at the boundary. Therefore, $U_k(t)$ must locally lie above $u(t)$ near $\psi = \hpi$ at all times. This implies
    $$R(u(\hpi,t)) \geq R(U_k(\hpi,t)) = \frac{2(1-u(\hpi,t)k^2)}{u(\hpi,t)} = \frac{2}{u(\hpi,t)} - 2k^2.$$
    Differentiating $\frac{1}{k^2} - u(\hpi,t)$ in $t$, we have
    \begin{align*}
        \frac{d}{dt}\parent{\frac{1}{k^2} - u\parent{\hpi,t}} &= u\parent{\hpi,t}R\parent{\hpi,t}\\
        &\geq u\parent{\hpi,t}\parent{\frac{2}{u(\hpi,t)} - 2k^2}\\
        &= 2k^2\parent{\frac{1}{k^2} - u\parent{\hpi,t}}.
    \end{align*}
    Integrating from time $t<0$ to $0$, one obtains
    $$\frac{1}{k^2} - u\parent{\hpi,t}  \leq \parent{\frac{1}{k^2} - u\parent{\hpi,0}}e^{2k^2t} = k^{-2}e^{2k^2t}. $$
    Now we prove the upper bound on $u(0,t)$ by proving that the spherical cap $U_k(\psi,t)$ actually lies nowhere below $u(\psi,t)$. Suppose this is not the case. Then, since $U_k(\psi,t)$ lies above $u(\psi,t)$ locally near $\psi=\hpi$, we can choose $k'$ such that $U_{k,k'}(t)\geq U_k(t)$ makes interior tangential contact with $u(t)$ from above. At that point, say $\psi=\psi_0$, one should have
    $$R(u(\psi_0,t)) \geq R(U_{k,k'}(t)) = \frac{2}{u(\hpi,t)} - 2k^2.$$
    Since our old solution satisfies $R_\psi\geq 0$, one must have
    $$R(u(\psi,t))\geq R(U_{k,k'}(t)) = \frac{2}{u(\hpi,t)} - 2k^2,\; \forall \psi \in [\psi_0,\hpi],$$
    meaning that $u(\psi,t)$ should actually locally lie below $U_{k}(t)$ for $\psi\in [\psi_0,\hpi]$. Appealing to {Lemma \ref{scalarcomp}} on the boundary, we are led to a contradiction. Therefore, $U_k(\psi,t) \geq u(\psi,t)$ for all $\psi$, from which follows the claimed bound
    $$u(0,t) \leq \frac{u(\hpi,t)}{\parent{1+k\sqrt{u(\hpi,t})}^2}.$$
    We prove the upper bound on $R(0,t)$ using a similar strategy. Define another spherical cap metric
    $$U_{low}(\psi,t) := \frac{u\parent{\hpi,t}}{\parent{1+\parent{\sqrt{\frac{u(\hpi,t)}{u(0,t)}}-1}\cos\psi}^2}.$$
    Observe that $$U_{low}(0,t) = u(0,t),\quad U_{low}(\hpi,t) = u(\hpi,t).$$ Also note that $U_{low}(t)$ is of constant scalar curvature
    $$R_{low}(\psi,t)\equiv R_{low}(t) = \frac{4}{\sqrt{u(0,t)u(\hpi,t)}} - \frac{2}{u(0,t)}$$
    with boundary geodesic curvature
    $$k_{low}(t) = \frac{1}{\sqrt{u(0,t)}} - \frac{1}{\sqrt{u(\hpi,t)}}>0.$$
    Suppose $U_{low}(t)$ lies above $u(t)$ locally near $\psi = 0$, i.e. $R_{low}(t) \leq R(0,t)$. By rotational symmetry, two metrics meet tangentially at $\psi=0$. Define, for $k'>0$,
    $$U_{low,k'}(\psi, t) = {\left( \sqrt{ \frac{2}{\sqrt{u(0, t) u(\pi/2, t)}} - \frac{1}{u(0, t)} + (k')^2 } + k' \cos\psi \right)}^{-2}.$$
    The metric $U_{low,k'}$ is a spherical cap metric of boundary curvature $k'>0$ and constant scalar curvature $R_{low}(t)$, so that $U_{low,k_{low}(t)}(t) = U_{low}(t).$ Modifying $k'$, we again derive a contradiction to the fact that $u(t)$ uniquely minimizes the scalar curvature at the pole $\psi=0$. Therefore, we conclude
    $$R(u(0,t)) \leq R_{low}(t) = \frac{4}{\sqrt{u(0,t)u(\hpi,t)}} - \frac{2}{u(0,t)}.$$
    To show the claimed upper bound on $R(u(0,t))$, it suffices to prove
    $$k \leq \frac{1}{\sqrt{u(0,t)}} - \frac{1}{\sqrt{u(\hpi,t)}}$$
    which follows from the lower bound on $u(0,t)$:
    $$\frac{1}{\sqrt{u(0,t)}} - \frac{1}{\sqrt{u(\hpi,t)}} \geq \frac{1+k\sqrt{u(\hpi,t)}}{\sqrt{u(\hpi,t)}} - \frac{1}{\sqrt{u(\hpi,t)}} = k.$$
\end{proof}

\subsection{Existence of the ancient solution}
    
    By {Lemma \ref{confloc}} and Lemma {\ref{curvloc}}, we know that both $u_i(t)$ and $R(g_i(t))$ attain their strict maximum on the boundary $\psi = \pi/2$ and strict minimum at the pole $\psi=0$. To take a smooth limit on our sequence of old solutions, we need a uniform curvature estimate for each old solution $g_i(t)$. To do that, we must first obtain a uniform $C^0$-estimate for old solutions that only depends on time, i.e.
    $$u_i(\psi,t) \geq C(t) >0$$
    using the monotonicity of isoperimetric ratio for our old solutions. This guarantees the uniform ellipticity of the Laplacians $\Delta_{g(t)} = u(t)^{-1}\Delta_0$ associated with the flow.
    \begin{lemma}[$C^0$-estimate on $u_i$]\label{C0bound}
        Let $g_i(t) = u_i(t)g_{0}$ denote the old solutions defined above. Then for each time $t_0<0$, we have the following $C^0$-estimate uniform in $i$:
        \begin{equation}
            \inf_{\psi,i}u_i(\psi,t_0) = \inf_{i}u_i(0,t_0) \geq C(t_0)>0.
        \end{equation}
    \end{lemma}
    \begin{proof}
        The first equality follows from Lemma \ref{confloc} and Lemma \ref{olddataproperties}.
        Recall the area formula of the metric $u_i(\cdot,t)g_0$:
        \begin{equation}\label{confarea}
            A_i(t) = 2\pi \int_0^{\pi/2} u_i(\psi,t)\sin\psi ~d\psi.
        \end{equation}
        First we present a pinching estimate on $u_i(\psi,t)$. Set $v = \log u_i$. By the positive scalar curvature condition $R_i\geq 0$, we have
        $$v_{\psi\psi} + \cot\psi ~v_\psi \leq 2$$
        which is equivalent to
        $$(\sin\psi\cdot v_\psi)_\psi \leq 2\sin\psi.$$
        Making use of the fact that $v_\psi(0,t) = 0$ (due to the rotational symmetry of the metric $u_i(\psi,t)g_{S^2}$), integrating the above inequality from $0$ to $\psi$ yields
        \begin{equation}\label{pinchingv}
            \sin\psi\cdot v_\psi(\psi,t) \leq 2(1-\cos\psi),
        \end{equation}
        which implies $$v_\psi(\psi,t) \leq 2\frac{1-\cos\psi}{\sin\psi} = 2\tan(\psi/2).$$
        Integrating the above equation again from $0$ to $\psi$, we have
        $$v(\psi,t) - v(0,t) \leq -4\log\cos (\psi/2),$$
        which yields a nice pinching estimate
        \begin{equation}\label{pinching}
            u_i(0,t) \leq u_i(\psi,t) \leq \cos^{-4}(\psi/2)\cdot u_i(0,t) \leq 4u_i(0,t)
        \end{equation}
        where the first inequality follows from {Lemma \ref{confloc}} and the last inequality from $\psi\in [0,\pi/2]$. In fact, via integrating the inequality (\ref{pinchingv}), we have a more general estimate
        \begin{equation}\label{pinchingratio}
            \frac{u_i(\psi,t)}{u_i(\psi',t)} \leq \parent{\frac{\cos (\psi/2)}{\cos (\psi'/2)}}^{-4}
        \end{equation}
        for $0\leq \psi'\leq \psi\leq \pi/2$.
        The relation (\ref{confarea}) gives us
        \begin{equation}\label{areaconfbound}
            A_i(t) \leq 2\pi\int_0^{\pi/2}u_i(0,t)\cos^{-4}(\psi/2)\sin\psi\;d\psi =  4\pi u_i(0,t).
        \end{equation}
        Therefore, to obtain the $i$ independent uniform lower bound on $u_i(0,t)$, it suffices to obtain a uniform-in-$i$ lower bound on the area $A_i(t)$.
        By the Gauss-Bonnet theorem
        $$\int_{\mathbb{D}}\frac{R}{2} dg_t + \int_{\partial \bD} k\;ds_t \equiv 2\pi = 2\pi \chi(\bD)$$
        one has
        \begin{equation}
            A_i'(t) = -\int_\bD R \; dg_t = -4\pi + 2kL_i(t).
        \end{equation}
        Since we know that $I_i(t)$ monotonically decreases for all $i$ to $2\pi$ by {Lemma \ref{isop}} and Lemma {\ref{olddataproperties}}, and the initial data of each old solution are all chosen very close to the flat disk, we can safely assume $2\pi \leq I_i(t) \leq 4\pi$, for all $i$ and $t$. Then we have the bounds
        \begin{equation}
            L_i(t) \leq \sqrt{4\pi A_i(t)}
        \end{equation}
        which imply
        \begin{equation}\label{odecomp}
            A_i'(t) \leq -4\pi + 2k\sqrt{4\pi A_i(t)}.
        \end{equation}
        Then, whenever $A_i(t) \leq \frac{\pi}{4k^2}$, 
        $$A_i'(t) \leq -2\pi.$$
        Integrating both sides from $t=t_0<0$ to $0$, we have
        $$A_i(t_0) \geq -2\pi t_0>0.$$
        Therefore, at $t=t_0$ we have a uniform lower bound
        \begin{equation}
            A_i(t_0) \geq \min\left\{2\pi|t_0|,\frac{\pi}{4k^2}\right\} =: 4\pi C(t_0).
        \end{equation}
        Combining with the previously obtained bound (\ref{areaconfbound}), we obtain the desired estimate.
    \end{proof}

    \begin{remark}
        The upper uniform $C^0$-bound follows easily from the pinching estimate (\ref{pinching}). Also note that the proof of {Lemma \ref{C0bound}} is essentially a type of noncollapsing estimate.
    \end{remark}

    \begin{lemma}[Uniform curvature estimate]\label{C0curvbound}
        $$\sup_{i,\psi}R_{g_i(t)}(\psi) = \sup_{i}R_{g_i(t)}(\pi/2) \leq C'(t) <\infty.$$
    \end{lemma}
    \begin{proof}
        By the Ricci flow equation, we have
        \begin{equation}\label{wieqn}
            \partial_t w_i(\psi,t) - w_i^3\nabla^0\cdot\parent{w_i^{-1}\nabla^0w_1} + w_i^3 = 0
        \end{equation}
        where $w_i = (u_i)^{-1/2}$. $w_i$'s satisfy the boundary conditions
        $$\partial_{\nu_0}w_i = -k.$$
        In addition, we also know that $\partial_\psi u_i$ are bounded uniformly-in-$i$: $\partial_\psi u_i\geq 0$ and the pinching estimate (\ref{pinchingratio}) yields
        $$0\leq \frac{u_i(\psi,t) - u_i(\psi',t)}{\psi - \psi'} \leq u_i(\psi',t)\cdot \frac{\frac{\cos^{-4}(\psi/2)}{\cos^{-4}(\psi'/2)} - 1}{\psi - \psi'}.$$
        Taking the limit $\psi \ra \psi'$, we obtain a uniform gradient estimate of form
        \begin{equation}\label{gradu}
            |\partial_\psi u_i(\psi',t)| = \partial_\psi u_i(\psi',t) \leq u_i(\psi',t)\cdot 2\tan(\psi'/2) \leq C(t)
        \end{equation}
        where $C(t)$ is a constant obtained from the uniform upper bound of $u_i$'s and independent of $i$, $\psi'$. Therefore, in any compact interval $I\subset (-\infty,0)$, the family of functions $w_i$ satisfy parabolic equations with oblique boundary condition, of uniformly bounded coefficients, which are guaranteed by the bounds of $w_i$ in any compact interval obtained in {Lemma \ref{C0bound}}. An application of parabolic De Giorgi-Nash-Moser estimate (cf. \cite{Lieberman1996}) on (\ref{wieqn}) yield
        $$||w_i||_{\delta,\frac{\delta}{2};K\times I} \leq C\big(\sup_{\old}|u_i|,\sup_{\old}|\nabla^0u_i|, k, I\big)$$
        for some $\delta\in (0,1)$, for any compact interior subset $K$ of $\old$. Since we have uniform bounds on $u_i$ and $\nabla^0 u_i = \partial_\psi u_i$, the right hand side of above inequality is in fact independent of $i$. Then $w_i$'s actually satisfy parabolic Neumann boundary problems with uniformly bounded parabolic H\"older coefficients, from which we obtain the global Schauder estimate
        $$||w_i||_{C^{2+\delta,1+\frac{\delta}{2}}(\old\times I)} \leq C(||w_i||_{\delta,\frac{\delta}{2};\old\times I} + ||w_i||_{\infty,I} + ||\nabla^0 w_i||_{\infty,I}, k) \leq C_0(I)$$
        where $C_0$ is a constant only depending on the uniform $C^0,C^1$ bounds of $u_i$'s on the compact interval $I$. Above uniform bound on the second order Schauder norms of $w_i$'s yields the uniform curvature bound of $g_i := u_ig_0$ on each compact time interval $I$.
    \end{proof}
    We can finally claim the existence of a rotationally symmetric ancient solution by showing that a limit flow $$g_\infty(t) := \lim_{i\ra \infty}g_i(t)$$
    exists and is also a smooth solution to the equation (\ref{2Dsym}) on $D\times (-\infty,0)$.
    \begin{proposition}[Existence]\label{existence}
        There exists a positively curved ancient solution $g_\infty(t)$ to the rotationally symmetric Ricci flow with boundary (\ref{2Dsym}) which smoothly converges to a flat disk of radius $1/k$ as $t\ra -\infty$ and has forward half-spherical singularity.
    \end{proposition}
    \begin{proof}
        Combine Lemma \ref{C0curvbound} and Shi's derivative estimate (cf. \cite{Gia16b}) to obtain uniform curvature derivative estimates of all order in every compact interval $I\subset (-\infty,0)$. Then Hamilton's compactness theorem (provided by \cite{CoMu19}; also see \cite{Gia16b}) yields an ancient flow which satisfies the same {Ricci flow with boundary} equation \eqref{2DRF}. The constructed ancient solution backward converges to a flat disk metric $g_{\fl}$ by the estimates in Proposition \ref{scalandconformal} and the monotonicity \eqref{2Du} of the conformal factor $u(\cdot, t)$ along the flow.
    \end{proof}

\subsection{Properties of ancient solutions converging to a flat disk.}\label{subsec_anyancient} Now let $\big(g(t) = u(\psi,t)g_0\big)_{\{t<0\}}$ denote \textit{any} rotationally symmetric ancient solution to {Ricci flow with boundary} (\ref{2DRF}) with the flat disk of radius $1/k$ as the smooth backward limit. We investigate the properties of $g(t)$. 

\begin{lemma}[Location of conformal factor/curvature extremum]\label{anysolngrad}
    Let $g(t) = u(\psi,t)g_0$ denote \textit{any} rotationally symmetric, positively curved ancient solution to (\ref{2Dsym}) with the flat disk of radius $1/k$ as the backward limit. Then for any $t<0$, we have
    $$u_\psi(\psi,t),\; R_\psi(\psi,t) \geq 0.$$
\end{lemma}
\begin{proof}
    Define the function $q$ is chosen as in the proof of {Lemma \ref{confloc}}: $$q(\psi,t) := \frac{u_\psi}{u\sin\psi}.$$
    Note that, as we have already assumed that the uniform smooth backward limit of $u(\psi,t)$ is a flat disk, we have the uniform convergence
    \begin{equation}\label{bconv}
        \lim_{t\ra -\infty} q(\psi,t) = \frac{2}{1+\cos\psi} \geq 1.
    \end{equation}
    Define an auxillary function
    $$W(\psi,t) = W_{\lda}(\psi,t) := q(\psi,t)e^{-\lda t},$$
    where $\lda>0$ is a constant to be determined. By (\ref{bconv}), we know that there exists a negative constant $T'>-\infty$ chosen independent of $\lda$ satisfying $W(\psi,t) >0$ for all $t\leq T'$. To prove that $u_\psi \geq 0$ for all time $t<0$, we shall argue by contradiction. Fix some $\tau<0$ for a moment. Suppose the contrary: say there exists some point $p_0 = (\psi_0,t_0) \in [0,\pi/2)\times (-T',\tau])$ such that $W(\psi_0,t_0) <0$. Note that the boundary point $\psi = \pi/2$ is not of concern as $W(\pi/2,t)$ remains positive throughout the flow. Also assume that $W(\psi_0,t_0)$ is a spatial minimum of $W$ at $t=t_0$. Choose 
    $$\lda = \lda_{\tau} := 1+4\sup_{(\psi,t)\in [0,\pi/2]\times (-\infty,\tau]}\frac{|q(\psi,t)|}{|u(\psi,t)|} <\infty.$$ 
    Then at the point $p_0=(\psi_0,t_0)\in [0,\pi/2)\times (T',0)$, $W$ satisfies
    \begin{align*}
        0&\geq W_t(\psi_0,t_0) \\
        &= q_t(p_0)e^{-\lda t_0} - \lda q(p_0)e^{-\lda t_0}\\
        &= e^{-\log u(p_0)}\big(\Delta_{S^4_+}W(p_0) - \sin\psi_0\cdot q(p_0)W_\psi(p_0) \\
        &\quad- 2\cos\psi_0\cdot q(p_0)W(p_0)\big) - \lda W(p_0)\\
        &\geq \bigg(\frac{-2\cos\psi_0\cdot q(p_0)}{u(p_0)} - \lda\bigg)W(p_0) >0
    \end{align*}
    which leads to a contradiction. Therefore, the function $W_{\lda(\tau)}(\psi,t)$ does not obtain negative parabolic minimum within the interval $[T',\tau] \subset (-\infty,0)$. Letting $\tau\ra 0$ (therefore $\lda(\tau) \ra \infty$ if necessary), we see that $q(\psi,t)$ cannot attain a negative parabolic minimum in all subintervals $I\subsetneq (-\infty,0)$. This means that $q(\psi,t)$ must always be nonnegative, which implies the nonnegativity of $u_\psi(\psi,t)$ for all $t<0$.
    
    For the scalar curvature derivative $R_\psi$, we define smooth functions
    $$q(\psi,t) := \frac{R_\psi}{\sin\psi},\; W_\lda^\eps(\psi,t) := e^{-\lda t}(q(\psi,t)+\eps)$$
    for some positive constants $\eps>0$ and $\lda = \lda_\eps>0$ to be assigned. Since we have prescribed the smooth backward limit, $q \ra 0$ uniformly as $t\ra -\infty$; and on the boundary $\psi = \pi/2$, $q(\pi/2,t) = R_\psi(\pi/2,t) \geq 0$ for all $t<0$, given by the boundary condition of our flow equation \eqref{2DRF}. Then for any $\eps>0$, we can choose $\tau_\eps \in (-\infty,-2)$ so that $|q(\psi,t)|\leq \eps/2$ for all $t<\tau_\eps$. Also, for $t<\tau_\eps$, $W^\eps_\lda(\psi,t) >0.$ Define $$\lda = \lda_{\eps} := 4\sup_{t\leq \tau_\eps^{-1}} |R(t)|.$$
    
    We argue as in the case of $u_\psi$: assume for the contrary that the function $W_\lda^\eps$ achieves a negative parabolic minimum 
    $$W_\lda^\eps(\psi_0,t_0) = \min_{[\tau_\eps,t_0]\times [0,\pi/2]} W_\lda^\eps <0$$
    on the interval $[\tau_\eps ,t_0] \subset I_\eps := [\tau_\eps,\tau_\eps^{-1}]$ for the first time.

    By the boundary condition of $R$, $\psi_0 \neq \pi/2$; the smooth rotational symmetry of the metric and the positive curvature condition yields $\psi_0\neq 0$. Therefore $(\psi_0,t_0)$ must be a strictly interior parabolic minimum point of $W_\lda^\eps$ in $[0,\pi/2]\times (-\infty,\tau^{-1}]$. Consequently, at $p_0 = (\psi_0,t_0)$ it holds
    \begin{align*}
        0\geq &~ e^{\lda t}\partial_t (e^{-\lda t}q)(p_0) = \partial_tq(p_0)-\lda q(p_0) \\
        = &~u(p_0)^{-2}\big(\Delta_{S^4_+}q(p_0) - 2q(p_0) -2\frac{u_\psi}{u}(p_0)(q_{\psi}(p_0) + 2q(p_0)\cot \psi_0) \big) \\
        &+ 2R(p_0)q(p_0)- \lda q(p_0)\\
        \geq &~ (2R(p_0) - \lda)q(p_0) >0,
    \end{align*}
    leading to contradiction. Here we have used the fact $u_\psi \geq 0$ proven above. Therefore any parabolic minimum of $W_\lda^\eps$ in finite interval cannot attain negative value within $I_\eps$. Now let $\eps \ra 0$. Then $\tau = \tau_\eps \ra -\infty$, $I_\eps \ra (-\infty,0)$ and we deduce that $q(\psi,t) \geq 0$ for all $\psi$ and $t\in (-\infty,0)$.
\end{proof}
Using the nonnegativity of $u_\psi$ and $R_\psi$, we also deduce the following upper bound on $u(0,t)$:
\begin{lemma}\label{anysolnbound}
For any ancient solution $g(t) = u(t)g_0$ to \eqref{2DRF} it holds
    $$u(0,t)\leq \frac{u(\hpi,t)}{\parent{1+k\sqrt{u(\hpi,t})}^2}.$$
\end{lemma}
\begin{proof}
    Identical to the proof of {Lemma \ref{scalandconformal}}.
\end{proof}

\section{Uniqueness of the flat disk backward limit}\label{flatbackwardlimit}
In this section, we prove that an ancient, positively curved Ricci flow $(\old,\partial \bD,g(t))_{t<0}$ with uniformly bounded diameter
\begin{equation}\label{bdddiam}
    \sup_{t<0} \text{diam} (\old,g(t))=\lim_{t\to-\infty}\text{diam} (\old,g(t)) \leq D<\infty
\end{equation}
must converge smoothly (up to the boundary) to a flat Euclidean disk $(\old,g_{\fl})$ of boundary curvature $k>0$ as $t\to -\infty$.
\begin{proposition}\label{flatdisklimit}
Let $(M^2, \partial M^2, g(t))$ be a compact ancient two-dimensional Ricci flow  with  boundary $\partial M^2$ such that  $(M^2, g(t))$ has nonnegative Gaussian curvature and bounded diameter $$\sup_{t}\operatorname{diam}(M^2, g(t))\leq D<+\infty,$$ and $\partial M^2$ has constant geodesic curvature $k>0$.  Then, up to a time independent diffeomorphism, the conformal factor $u$ smoothly converges to the standard conformal factor 
$\frac{1}{k^2(1+\cos \psi)^2}$ of the flat
disk $(\old, g_{\fl})$ whose boundary has constant geodesic curvature $k$.

\end{proposition}
By \cite{CoMu19}, we know that the flow $g(t)$ is parametrized by a conformal factor $u(t)g_0 = g(t)$ as well. Recall the Gauss-Bonnet formula
\begin{equation}\label{gaussbonnet}
    \int_\bD R\, dA_{g(t)} + 2k\int_{\partial\bD} d\sigma_{g(t)} = 4\pi\chi(\bD) = 4\pi
\end{equation}
where $dA_{g(t)}$ denotes the area element and $d\sigma_{g(t)}$ the length element on the boundary $\partial\bD$, both induced by $g(t)$. Under the Ricci flow, it follows trivially that the area of the disk $(\bD,g(t))$, $A(t) := \int_\bD dA_{g(t)}$ evolves by the equation
$$\frac{\partial}{\partial t} A(t) = -\int_\bD R\, dA_{g(t)} <0.$$
We first prove that the bounded diameter condition \eqref{bdddiam} plus the positive curvature condition forces the area $A(t)$ to converge backward to a finite number. Denote the length function $L(t) := \int_{\partial\bD} d\sigma_{g(t)} $.
\begin{lemma}\label{QL convergence}
    For all $t<0$, 
    \begin{equation}\label{Abound}
        A(t) \leq D\cdot L(t) \leq \frac{2\pi}{k}D.
    \end{equation}
    Moreover we have
    \begin{equation}\label{curvL1decay}
        \lim_{t\to -\infty}Q(t) := \int_{\bD} R\, dA_{g(t)} =0,\quad \lim_{t\to -\infty}L(t) =\frac{2\pi}{k}.
    \end{equation}
\end{lemma}
\begin{proof}
    The second inequality in \eqref{Abound} follows from the Gauss-Bonnet formula \eqref{gaussbonnet} and $R\geq 0$; the second inequality in \eqref{curvL1decay} follows from the first inequality with the Gauss-Bonnet formula. Therefore, it remains to prove the first inequalities in \eqref{Abound} and \eqref{curvL1decay}. For $y\in \partial\bD$, let $N_y$ denote the inward unit normal and let
    $$\Gamma(y,s) = \Gamma_y(s) := \exp_s(sN_y), \; 0\leq s\leq c(y) $$
    where $c(y)$ denotes the boundary cut distance in the direction of $N_y$. Then the boundary normal coordinate maps $(\Gamma_y)_{y\in \partial\bD}$ cover $\bD$ exactly once away from a measure zero set. Pulling back by the coordinate map $\Gamma$, we have
    $$\Gamma^*(dA_{g(t)}) = J(\cdot,s)ds\,d\sigma_{g(t)}(\cdot)$$
    where $J_y(s) = J(y,s)$ is the scalar Jacobi field satisfying
    $$J_y'' + \frac{1}{2}R(\Gamma_y(s))J_y = 0, \; J_y(0)=1,\; J_y'(0) = -k.$$
    Since $J_y'' \leq 0$ and $J_y'(0) <0$, we know that $J_y(s) \leq 1$. Also note that the cut distance $c(y) \leq D = \text{diam}\, (\bD,g(t)).$
    Then the standard coarea formula yields
    $$A(t) = \int_{\partial\bD}\int_0^{c(y)} J(y,s)\; dsd\sigma_{g(t)}(y) \leq D\int_{\partial\bD} d\sigma_{g(t)} \leq DL.$$
    Now we prove the first inequality in \eqref{curvL1decay}. Differentiating the Gauss-Bonnet formula \eqref{gaussbonnet} in $t$, we have
    $$-Q'(t) + 2kL'(t) = -Q'(t) + k\int_{\partial\bD} R\, d\sigma_{g(t)} = 0,$$
    yielding
    $$Q'(t) \geq 0.$$
    Then
    $$\frac{2\pi}{k}D \geq A(t_0) - A(t) = \int_{t_0}^t Q(s)\,ds \geq 0 $$
    for all $t_0<t$. Sending $t_0\to -\infty$, we find that $Q(s) \to 0$ as $s\to -\infty$.
\end{proof}
Now using a standard point-picking argument, we improve the above $L^1$-decay of curvature to a $L^\infty$-decay, i.e.
$$\lim_{t\to-\infty} R_{max}(t) = 0.$$
\begin{lemma}\label{pointwisedecay}
    Under the above assumptions,
    $$\lim_{t\to-\infty} R_{max}(t) = 0.$$
\end{lemma}
\begin{proof}
    Suppose the above is false. Then there exists a constant $\eps>0$ and a sequence of time $t_j\to -\infty$ such that
    $$\lda_j := R_{max}(t_j) \geq \eps,\; \forall j.$$
    Without loss of generality, for all $j$, assume that
    $$\sup_{s\in [t_j-\lda_j^{-1},t_j]} R_{max}(s) \leq 2\lda_j.$$
    If not, we can choose another time in a smaller subinterval iteratively. Note that this time-picking process must cease in a finite number of iterations due to the smoothness of the ancient flow. Now choose points $p_j$ satisfying
    $$\lda_j = R(p_j,t_j).$$
    Then we have a sequence $\{(\lda_j,p_j,t_j)\}$ with
    $$\lda_j\geq \eps, \; R(x,t) \leq 2\lda_i,\;\forall t\in [t_j - \lda_j^{-1},t_j] \to -\infty.$$
    Define the parabolically rescaled flows
    $$g_j(t) := \lda_jg(t_i + \lda_j^{-1}t),\; t\in [-1,0].$$
    The rescaled flows still satisfy the Ricci flow equation, and satisfy uniform bounds
    $$R_{g_j} \in (0,2],\; R_{g_j}(p_j,0) = 1,\; k_{g_j} \equiv \frac{k}{\lda_j} \leq \frac{k}{\eps}$$
    where $k_{g_j}$ denotes the (constant) boundary geodesic curvature of $g_j(t)$.
    Moreover, using the scale invariance of the total curvature, we find that
    $$\int_\bD R_{g_j(0)} \; dA_{g_j(0)} = Q(t_j) \to 0 \text{ as }j\to \infty$$
    by {Lemma \ref{QL convergence}}. Finally, note that the boundary length $L(g_j(0))$ of $g_j(0)$ has the uniform lower bound
    $$L(g_j(0)) = \sqrt{\lda_j}L(t_j) \geq \sqrt{\eps}\frac{\pi}{k}$$
    for sufficiently large $j$. We now invoke Hamilton's compactness theorem (\cite[Theorem 4.1]{Gia16b}) to take a pointed subsequential smooth limit
    $$(\old, g_j(t),p_j) \to (\bD_\infty,g_\infty(t),p_\infty)$$
    on compact spaces for $t\in [0,1]$. Note that the limit space $\bD_\infty$ may not be a closed disk, or even a smooth manifold with boundary. By the rescaling,
    $$R_\infty(p_\infty,0) =1.$$
    Choose a small neighborhood $U$ (a coordinate ball if $p_\infty \in \text{Int}\bD_\infty$; a half-ball if $p_\infty\in \partial\bD_\infty$) of $p_\infty \in \bD_\infty$ so that $R_\infty(0) >0$ in $U$. Then we have
    $$0<\int_U R_\infty \; dA_{g_\infty(0)} \leq \lim_{j\to \infty} \int_{\old} R(g_j(0)) \,dA_{g_j(0)} = 0,$$
    which leads to a contradiction.
\end{proof}

\begin{proof}[Proof of Proposition \ref{flatdisklimit}]
    Consider the time-shifted flows
    $$h_j(t) := g(t_j+t), \; t\in [-1,0].$$
    Then the previous lemma yields
    $$\sup_{t\in [-1,0]} R(h_j(t)) \to 0 \text{ as } j\to\infty.$$
    Also by {\eqref{curvL1decay}} and Lemma {\ref{pointwisedecay}}, we know that the flows $h_j(t)$ satisfy a uniform lower bound on boundary curve length and upper bound on the diameter. Invoking Hamilton's compactness theorem (\cite{Gia16b} and \cite{CoMu19}), we obtain a subsequential limit flow $(\bD_\infty,h_\infty(t))$ (in the sense of Cheeger-Gromov-Hamilton)
    $$h_\infty(t) := \lim_{j\to\infty} h_j(t), \; t\in [-1,0]$$
    which is stationary. As the compactness theorem preserves the boundary curvature condition, we deduce that the subsequential limit
    $$h_\infty := \lim_{j\to \infty} h_j(0) = \lim_{j\to \infty} g(t_j)$$
    is a stationary flat disk of constant boundary curvature $k_\infty \equiv k>0$. This proves the existence of a smooth subsequential limit. Due to the monotonicity of the conformal factor provided by the flow equation (\ref{2Du}), prescribing a subsequential backward limit automatically yields the full convergence of conformal factor $u(\cdot, t)$ for all time $t$ going to $-\infty$. As all Riemannian metrics on $\bD$ are conformally equivalent, we know that the limit $\lim_{t\to -\infty} u(t)$
    is uniquely determined and equivalent to a standard conformal factor of flat Euclidean metric $(\old, g_{\fl})$ after a suitable diffeomorphism (see also Section \ref{subsec_reduction}).
\end{proof}

\section{Rotational symmetry}\label{symmetrysection}

In this section we prove the rotational symmetry of any ancient solution to the Ricci flow with boundary equation \eqref{2DRF} which converges smoothly as $t\rightarrow -\infty$ to a flat metric on $\old$ with constant geodesic curvature $k>0$. Recall that by Proposition \ref{flatdisklimit}, any positively curved ancient solution of \eqref{2DRF} with uniformly bounded diameter does indeed satisfy this backward convergence property. The main result of this section is as follows.

\begin{proposition}\label{prop_g_symmetry}
If $(\old,\partial \bD,g(t))$ is an ancient solution of \eqref{2DRF} which converges in $C^\infty(\old)$ as $t\rightarrow -\infty$ to a flat metric on $\old$ with constant geodesic curvature $k>0$, then there is a diffeomorphism $\Psi:\old\rightarrow\old$ such that $\Psi^*(g(t))$ is rotationally symmetric for all $t<0$.
\end{proposition}

As discussed in Section \ref{subsec_reduction}, Proposition \ref{prop_g_symmetry} will follow immediately from Lemma \ref{lem_scalar_reduction} and Proposition \ref{prop_symmetry}. The purpose of Sections \ref{subsec_reflection} and \ref{subsec_backwards} is to prove Proposition \ref{prop_symmetry}.

\subsection{Reduction to a scalar evolution equation}\label{subsec_reduction}

To analyze the ancient solutions considered in Proposition \ref{prop_g_symmetry}, it will be convenient to work with respect to the fixed coordinate system on $\old$ and the flat background metric $g_{\mathrm{flat}}$ induced by the standard inclusion $\old\subset\mathbb{R}^2$. That we can do so follows by standard arguments (see for instance \cite[\S2]{WangX}), but below we restate the result and its proof formulated in our setting for clarity.

\begin{lemma}\label{lem_scalar_reduction}
If $(\old,\partial \bD,g(t))$ is an ancient solution of \eqref{2DRF} which converges in $C^\infty(\old)$ as $t\rightarrow -\infty$ to a flat metric on $\old$ with constant geodesic curvature $k>0$, then there is a diffeomorphism $\Psi:\old\rightarrow\old$ and a smooth function $u(\cdot,t)$ such that
\[
\Psi^*(g(t))=u(\cdot,t)g_{\mathrm{flat}},
\]
with $u(\cdot,t)\xrightarrow{t\rightarrow-\infty}1$ smoothly on $\old$.
\end{lemma}
\begin{proof}
Given an ancient Ricci flow with boundary $g(t)$ satisfying the hypotheses of the lemma, fix some time $t_0<0$. Then by the uniformization theorem, there exists a diffeomorphism $H:\old\rightarrow\old$ so that $H^*g(t_0)=\tilde{u}(\cdot,t_0)g_{\mathrm{flat}}$ for some smooth function $\tilde{u}(\cdot,t_0)>0$. Since $H^*g(t)$ also satisfies the Ricci flow equation \eqref{2DRF}, we can write $H^*g(t)=\tilde{u}(\cdot,t)g_{\mathrm{flat}}$ for some smooth $\tilde{u}(\cdot,t)$ and all $t$, with the smooth convergence $\tilde{u}(\cdot,t)\xrightarrow{t\rightarrow-\infty}\tilde{u}_\infty(\cdot)$ such that $\tilde{u}_\infty(\cdot)g_{\mathrm{flat}}$ is flat with constant unit boundary geodesic curvature. There then exists a diffeomorphism given by a fractional linear transformation $F:\old\rightarrow\old$ such that $g_{\mathrm{flat}}=F^*(u_\infty(\cdot)g_{\mathrm{flat}})$, and since $F$ is a conformal map this means there is some smooth function $u(\cdot,t)$ such that for all $t<0$,
\[
F^*(H^*g(t))=F^*(\tilde{u}(\cdot,t))g_{\mathrm{flat}})=u(\cdot,t)g_{\mathrm{flat}}.
\]
Clearly $u(\cdot,t)\xrightarrow{t\rightarrow-\infty}1$ smoothly, so we can take $\Psi=H\circ F$.
\end{proof}

Thus, to prove Proposition \ref{prop_g_symmetry} it suffices to prove the rotational symmetry of the function $u(\cdot,t)$ in Lemma \eqref{lem_scalar_reduction}. We now make a transformation which will be useful for this purpose. Given a solution $u(\cdot,t) g_{\mathrm{flat}}$ of \eqref{2DRF} as in Lemma \ref{lem_scalar_reduction}, define the profile function $v=\frac{1}{\sqrt{u}}-1$, as we did in \eqref{eq_profile} but now with respect to the flat rather than spherical metric on $\old$ as the background metric. We then find that $v$ satisfies
\begin{equation}\label{eq_symmetry_v}
\begin{cases}
v_t=(1+v)^{-2}\Delta v-(1+v)^{-3}|\nabla v|^2\quad&\text{in }\old\times(-\infty,0),
\\
\partial_r v=v\quad&\text{on }\partial \bD\times(-\infty,0),
\end{cases}
\end{equation}
with $v(\cdot,t)$ smoothly converging to $0$ as $t\rightarrow -\infty$. Proposition \ref{prop_g_symmetry} therefore follows from our discussion above and Proposition \ref{prop_symmetry} below.

\begin{proposition}\label{prop_symmetry}
If $v:\old\times(-\infty,0)\rightarrow\bR$ satisfies \eqref{eq_symmetry_v} with $v(\cdot,t)\xrightarrow{C^\infty(\old)} 0$ as $t\rightarrow -\infty$, then $v(\cdot,t)$ is rotationally symmetric for all $t<0$.
\end{proposition}

The rest of this section is devoted to proving Proposition \ref{prop_symmetry}. We set up a reflection argument, establish a backward decay estimate for $v$, and finally conclude with the proof of Proposition \ref{prop_symmetry} at the end of Section \ref{subsec_backwards}. 

\subsection{A reflection argument}\label{subsec_reflection}

To prove Proposition~\ref{prop_symmetry} it suffices to show that $v(\cdot,t)$ is equal to its reflection over any diameter of $\old$, and by the $SO(2)$-invariance of $g_{\mathrm{flat}}$ it suffices to show this for the diameter $\{x_1=0\}$ under the convention that $(x_1,x_2)$ denotes the standard Euclidean coordinates on $\old$. For this purpose, define the reflection
\[
R:\old\longrightarrow\old,\qquad (x_1,x_2)\mapsto(-x_1,x_2),
\]
and let $v_R(x,t)=v(Rx,t)$. We would like to show that $v-v_R\equiv 0$. Below we first derive the parabolic equation satisfied by a related function.

\begin{lemma}\label{lem_neumann}
If $v:\old\times(-\infty,0)\rightarrow\bR$ satisfies \eqref{eq_symmetry_v}, then the function
\[
\phi:\old\longrightarrow\old,\qquad \phi(x,t)=\frac{v(x,t)-v_R(x,t)}{x_1}
\]
is even in the $x_1$ coordinate and satisfies
\begin{equation}\label{eq_psi}
\begin{cases}
\phi_t=A\Delta\phi+B\cdot\nabla\phi+C\phi\quad&\text{in }\old\times(-\infty,0),
\\
\partial_r\phi=0\quad&\text{on }\partial \bD\times(-\infty,0),
\end{cases}
\end{equation}
where
\begin{align*}
A&=(1+v)^{-2},
\\
B&=2(1+v)^{-2}\frac{\nabla x_1}{x_1}-(1+v)^{-3}\nabla(v+v_R),
\\
C&=\frac{\alpha(v)-\alpha(v_R)}{v-v_R}\Delta v_R+\frac{\beta(v)-\beta(v_R)}{v-v_R}|\nabla v_R|^2-(1+v)^{-3}\nabla(v+v_R)\cdot\frac{\nabla x_1}{x_1}.
\end{align*}
\end{lemma}
\begin{proof}
The results follow by direct computations which we detail below.

First, note that $v(x,t)-v_R(x,t)$ is a smooth function which is odd in $x_1$, so by Hadamard's lemma $\phi$ is smooth and well-defined on $\old$. Now let $\alpha(v)=(1+v)^{-2}$ and let $\beta(v)=-(1+v)^{-3}$. Recalling that $\Delta x_1=0$, we now compute directly that
\begin{align*}
(v-v_R)_t&=\alpha(v)\Delta v-\alpha(v_R)\Delta v_R+\beta(v)|\nabla v|^2-\beta(v_R)|\nabla v_R|^2
\\
&=\alpha(v)\Delta(v-v_R)+\beta(v)(\nabla v+\nabla v_R)\cdot \nabla(v-v_R)
\\
&\qquad\qquad + (\alpha(v)-\alpha(v_R))\Delta v_R+(\beta(v)-\beta(v_R))|\nabla v_R|^2.
\end{align*}
If we let $z=v-v_R$ then we can rewrite this as
\begin{equation}\label{eq_z}
z_t=\alpha(v)\Delta z+\beta(v)
\nabla(v+v_R)\cdot\nabla z+ F z,
\end{equation}
where
\[
F=\frac{\alpha(v)-\alpha(v_R)}{v-v_R}\Delta v_R+\frac{\beta(v)-\beta(v_R)}{v-v_R}|\nabla v_R|^2.
\]
Note that $F$ is well defined and smooth since $\alpha(v)$, $\beta(v)$ are smooth functions of $v$, as we always have $1+v>0$.

We then deduce \eqref{eq_psi} from \eqref{eq_z} from the identities
\[
\Delta z=x_1\Delta\phi+2\nabla x_1\cdot\nabla\phi,\qquad \nabla z=x_1\nabla\phi+\phi\nabla x_1.
\]

Finally, obtain the Neumann boundary data, first recall that $\partial_r x_1=x_1$ on $\partial \bD$, and then compute at $\partial \bD\cap\{x_1\neq 0\}$ that
\begin{align*}
\partial_r \phi(x,t)&=\partial_r\left(\frac{v(x,t)-v_R(x,t)}{x_1}\right)
\\
&=\frac{\partial_r v(x,t)-\partial_r v_R(x,t)}{x_1}-\frac{v(x,t)-v_R(x,t)}{x_1^2}\partial_r x_1
\\
&=\frac{v(x,t)-v_R(x,t)}{x_1}-\frac{v(x,t)-v_R(x,t)}{x_1}
\\
&=0.
\end{align*}
So by continuity we conclude that $\partial_r\phi=0$ on all of $\partial \bD$.
\end{proof}

\begin{remark}
The coefficient $B$ in \eqref{eq_psi} is singular on $\old\cap\{x_1=0\}$, due to $2(1+v)^{-2}\frac{\nabla x_1}{x_1}$. However, both sides of the parabolic equation \eqref{eq_psi} itself are well-defined and smooth. This is because $\frac{\nabla x_1}{x_1}\cdot\nabla\phi=\frac{\partial_{x_1}\phi}{x_1}$, which is smooth since $\phi$ is even in $x_1$. For a similar reason we see that the cofficient $C$ in \eqref{eq_psi} is smooth despite the $(1+v)^{-3}\nabla(v+v_R)\cdot\frac{\nabla x_1}{x_1}$ term, because $v+v_R$ is even in $x_1$.
\end{remark}

Applying the maximum principle to \eqref{eq_psi}, we obtain an $L^\infty$ estimate for the function $\phi$. Once we obtain a suitable decay estimate for $v$ as in Section \ref{subsec_backwards}, this will allow us to conclude that $\phi$ must vanish identically.

\begin{lemma}\label{lem_integral_estimate}
If $v:\old\times(-\infty,0)\rightarrow\bR$ satisfies \eqref{eq_symmetry_v}  with $v(\cdot,t)\xrightarrow{C^\infty(\mathbb{D})} 0$ as $t\rightarrow -\infty$, and we define $\phi(x,t)=\frac{v(x,t)-v_R(x,t)}{x_1}$ as in Lemma \ref{lem_neumann}, then there exists some $T<0$ and $\gamma>0$ such that for any $-\infty<s<t\leq T$ we have
\begin{equation}\label{eq_integral_estimate}
\|\phi(\cdot,t)\|_{L^\infty(\old)}\leq \gamma\left(\int_s^t \|v(\cdot,\tau)\|_{C^2(\old)}~d\tau\right)\|\phi(\cdot,s)\|_{L^\infty(\old)}.
\end{equation}
\end{lemma}
\begin{proof}
For any fixed $\tau\in[s,t]$, suppose that $\phi(\cdot,\tau)$ achieves a maximum at $p_\tau\in\old$. We first check using \eqref{eq_psi} and the notation of Lemma \ref{lem_neumann} that $\phi_t(p_\tau)\leq (C\phi)(p_\tau)$:
\begin{enumerate}
\item If $p_\tau\in\mathrm{Int}(\old)\cap\{x_1\neq 0\}$, then $\Delta\phi(p_\tau)\leq 0$ and $\nabla\phi(p_\tau)=0$, so $(A\Delta\phi)(p_\tau)+(B\cdot\nabla\phi)(p_\tau)\leq 0$.
\item If $p_\tau\in\mathrm{Int}(\old)\cap\{x_1= 0\}$, then again $(A\Delta\phi)(p_\tau)\leq 0$. We also have that $(B\cdot\nabla\phi)(p_\tau)=\left(2(1+v)^{-2}\phi_{x_1x_1}\right)(p_\tau)\leq 0$.
\item If $p_\tau\in\partial \bD\cap\{x_1\neq 0\}$, then the Neumann condition $\partial_r\phi(p_\tau)=0$ implies that $\partial_r^2\phi(p_\tau)\leq 0$, while we also have $\partial_\theta\phi(p_\tau)$ and $\partial^2_\theta\phi(p_\tau)\leq 0$. So again $\Delta\phi(p_\tau)\leq 0$ and $\nabla\phi(p_\tau)=0$ and $(A\Delta\phi)(p_\tau)+(B\cdot\nabla\phi)(p_\tau)\leq 0$.
\item If $p_\tau\in\partial \bD\cap\{x_1= 0\}$, then again $\partial^2_r\phi(p_\tau)\leq 0$ and $\partial^2_\theta\phi(p_\tau)\leq 0$, so that $(A\Delta\phi)(p_\tau)\leq 0$. We also have that $(B\cdot\nabla\phi)(p_\tau)=\left(2(1+v)^{-2}\phi_{\theta\theta}\right)(p_\tau)\leq 0$
\end{enumerate}
The same argument also works with $\phi$ replaced by $-\phi$, so Hamilton's trick \cite[Lemma 3.5]{Hamilton86} (see \cite[Lemma 2.1.3]{Mantegazza} for a version more closely related to our setting) then implies that 
\[
\|\phi(\cdot,t)\|_{L^\infty(\old)}\leq \left(\int_s^t \max_{\old}\|C(\cdot,\tau)\|_{L^\infty(\old)}~d\tau\right)\|\phi(\cdot,s)\|_{L^\infty(\old)}.
\]
To conclude, we observe from the convergence $v(\cdot,t)\xrightarrow{C^\infty(\old)} 0$ as $t\rightarrow -\infty$ that we can find constants $T<\infty$ and $\gamma>0$ so that for all $\tau\leq T$ we have
\[
\|C(\cdot,\tau)\|_{L^\infty(\old)}\leq\gamma\|v(\cdot,\tau)\|_{C^2(\old)}.
\]
\end{proof}

\subsection{Backward exponential convergence of the conformal factor}\label{subsec_backwards}

We now estimate the size of $\|v(\cdot,\tau)\|_{C^2(\old)}$ when $\tau$ is sufficiently negative. The decay rate estimate we obtain below for $v$ will allow us to conclude using Lemma \ref{lem_integral_estimate} that $\phi$ must vanish identically, which will prove Proposition \ref{prop_symmetry}.

\begin{lemma}\label{lem_decay_estimate}
If $v:\old\times(-\infty,0)\rightarrow\bR$ satisfies \eqref{eq_symmetry_v} with $v(\cdot,t)\xrightarrow{C^\infty(\mathbb{D})} 0$ as $t\rightarrow -\infty$, then there exists $S<0$ and $a,b>0$ such that for all $t\leq S$ we have
\begin{equation}\label{eq_decay_estimate}
\|v(\cdot,t)\|_{C^2(\old)}\leq a e^{bt}.
\end{equation}
\end{lemma}
\begin{proof}
We will apply \cite[Theorem 4]{Feehan-Maridakis} in the case where the hypotheses of \cite[Theorem 2]{Feehan-Maridakis} are also satisfied. We begin by verifying these hypotheses are satisfied, in the notation of those results.

We define
\[
\mathscr{X}=\{v\in H^2(\bD):~\partial_r v=v\text{ on }\partial\bD\},\quad\tilde{\mathscr{X}}=L^2(\bD),
\]
where the Robin boundary condition defining $\mathscr{X}$ is well-defined in the trace-sense. Note that $\mathscr{X}$ is a closed subspace of the Banach space $H^2(\bD)$, and therefore also a Banach space (with the same norm). Therefore we have the Banach space embeddings
\[
\mathscr{X}\subset\tilde{\mathscr{X}}\subset\mathscr{X}^*,
\]
where the second embedding holds since $f\in L^2(\bD)$ acts on $v\in\mathscr{X}$ by $v\mapsto\int_{\bD} fv$. Moreover, the embedding $\mathscr{X}\subset\mathscr{X}^*$ is definite in the sense defined in \cite[p. 4]{Feehan-Maridakis} because the bilinear form $\mathscr{X}\times\mathscr{X}\rightarrow\mathbb{R}$, $(u,u)\mapsto\int_{\bD} u^2$ is positive-definite. Next we define the open subset $\mathscr{U}\subset\mathscr{X}$ by
\[
\mathscr{U}=\{v\in\mathscr{X}:~\|v\|_{H^2(\mathbb{D})}<\rho\},
\]
for a $\rho>0$ chosen small enough so that $\|v\|_\infty<\frac{1}{2}$ by Sobolev embedding. Therefore the function
\[
\mathscr{E}:\mathscr{U}\longrightarrow\mathbb{R},\qquad v\mapsto\frac{1}{2}\int_{\bD}\frac{|\nabla v|^2}{(1+v)^2}-\int_{\partial\bD}\left(\frac{1}{1+v}+\log(1+v)-1\right)
\]
is well-defined and real-analytic and thus a fortiori $C^2$, since the restriction $\|v\|_\infty<\frac{1}{2}$ allows us to expand $\frac{1}{(1+v)^2}$, $\frac{1}{1+v}$, and $\log(1+v)$ as uniformly convergent power series. We can then compute that the Fr\'{e}chet derivative of $\mathscr{E}$ at $v\in\mathscr{U}$ is given by
\begin{align*}
\mathscr{E}'(v)[h]&=\int_{\bD}\frac{\nabla v}{(1+v)^2}\cdot\nabla h-\frac{|\nabla v|^2}{(1+v)^3}h-\int_{\partial\bD}\frac{v}{(1+v)^2}h
\\
&=\int_{\bD}\left(-(1+v)^{-2}\Delta v+(1+v)^{-3}|\nabla v|^2\right) h.
\end{align*}
This implies that
\[
\mathscr{M}:\mathscr{U}\longrightarrow\tilde{\mathscr{X}},\qquad v\mapsto -(1+v)^{-2}\Delta v+(1+v)^{-3}|\nabla v|^2
\]
is a gradient map with potential $\mathscr{E}$, in the sense of \cite[Definition 1.4]{Feehan-Maridakis}, having used Sobolev embedding to ensure that the image of $\mathscr{U}$ indeed lies in $\tilde{\mathscr{X}}$. Moreover, $\mathscr{M}$ is real-analytic on its domain $\mathscr{U}$ because again we can expand $(1+v)^{-2}$ and $(1+v)^{-3}$ as uniformly convergent power series. Now we take $x_\infty\in\mathscr{U}$ to be the function $v\equiv 0$, which clearly satisfies $\mathscr{M}(x_\infty)=0$. We therefore have that
\[
\mathscr{M}'(x_\infty):\mathscr{X}\longrightarrow\tilde{\mathscr{X}},\qquad h\mapsto -\Delta h,
\]
and we recall that $h\in\mathscr{X}$ satisfies the Robin boundary condition $\partial_r h=h$. In particular by standard arguments $\mathscr{M}'$ is Fredholm with Fredholm index $0$. Moreover,
we have a basis of $L^2(\bD)$ given by eigenfunctions $\{f_i\}_{i=1}^\infty$ satisfying
\begin{equation}\label{eq_robin_eigenvalues}
\begin{cases}
\Delta f_i=\mu_i f_i\quad&\text{in }\old,
\\
\partial_r f_i=f_i\quad&\text{on }\partial \bD,
\end{cases}
\end{equation}
with $\mu_0>\mu_1\geq\mu_2\geq\mu_3\geq\cdots$ just as we discussed before in Section \ref{subsec_eigenfunction}, since this is the same eigenvalue problem as \eqref{critrobin}. The last hypothesis we need to check is that $\mathscr{E}$ is Morse--Bott at $x_\infty\in\mathscr{U}$ in the sense of \cite[Definition 1.10]{Feehan-Maridakis}. This is the content of Claim \ref{claim_morse_bott} below.
\begin{claim}\label{claim_morse_bott}
After possibly shrinking $\mathscr{U}$, we have that the critical set $\mathrm{Crit}\mathscr{E}=\{x\in\mathscr{U}:~\mathscr{E}'(x)=0\}$ is an open smooth submanifold of $\mathscr{U}$ and
\[
(T\mathrm{Crit}\mathscr{E})_{x_\infty}=\ker\mathscr{E}''(x_\infty)
\]
\end{claim}
\begin{proof}
Recall from Section \ref{subsec_eigenfunction} that the eigenvalues $\mu_i$ of \eqref{eq_robin_eigenvalues} satisfy
\[
\mu_0>0,\quad\mu_1=\mu_2=0,\quad\mu_i<0\text{ for all }i\geq 3,
\]
and the $\mu=0$ eigenspace is spanned by the two Euclidean coordinate functions on $\old$; hence $\ker\mathscr{E}''(x_\infty)=\mathrm{span}\{x_1,x_2\}$. On the other hand, from the automorphisms of $\bD\subset\mathbb{C}$ we find that the family of functions $q_a$ parametrized by $a\in\mathrm{Int}(\old)$ and defined by
\[
q_a:\bD\longrightarrow\mathbb{R},\qquad x\mapsto \frac{-2a\cdot x+|a|^2(1+|x|^2)}{1-|a|^2}
\]
belong to $\mathrm{Crit}\mathscr{E}$ whenever $q_a\in\mathscr{U}$, and the set $\mathscr{U}\cap\{q_a\}$ is an open smooth submanifold of $\mathscr{U}$ tangent to $\ker\mathscr{E}''(x_\infty)$ at $x_\infty$ if we take $\mathscr{U}$ sufficiently small.

So it remains only to show that $\mathscr{U}\cap\mathrm{Crit}\mathscr{E}=\mathscr{U}\cap\{q_a\}$, if $\mathscr{U}$ is taken small enough. Define $K=\ker\mathscr{E}''(x_\infty)$ and let $K^\perp$ be its $L^2(\bD)$ orthogonal complement, so that with $H=\mathscr{X}\cap K^\perp$ we have the direct sum decomposition $\mathscr{X}=K\oplus H$. Now consider the map
\[
F:K\times H\longrightarrow K^\perp,\qquad (k,h)\mapsto P_{K^\perp}(\mathscr{M}(k+h)),
\]
where $P_{K^\perp}$ denotes the orthogonal projection in $L^2(\bD)$ onto $K^\perp$. We therefore have that
\[
D_hF(0,0)=-\Delta|_{H}:H\rightarrow K^\perp
\]
is invertible, so by the implicit function theorem if we shrink $\mathscr{U}$ appropriately then there exists a smooth map $G:K\cap\mathscr{U}\rightarrow H$ so that
\[
\mathscr{U}\cap\mathrm{Crit}\mathscr{E}=\{k+h\in\mathscr{U}:F(k+h)=0\}=\{k+G(k)\in\mathscr{U}\}.
\]
In particular, the last equality implies that $\mathscr{U}\cap\mathrm{Crit}\mathscr{E}=\mathscr{U}\cap\{q_a\}$ as desired
\end{proof}
We are now in a position to apply the results of \cite{Feehan-Maridakis} using the definitions set up above. Let $v:\old\times(-\infty,0)\rightarrow\mathbb{R}$ satisfy \eqref{eq_symmetry_v}, with $v(\cdot,t)\xrightarrow{C^\infty(\old)} 0$ as $t\rightarrow -\infty$. First, note in the context of our new notation that our original evolution equation \eqref{eq_symmetry_v} says that $v_t=-\mathscr{M}(v)$, where $\mathscr{M}$ is the gradient of the energy $\mathscr{E}$. So we also have
\begin{equation}\label{eq_energy_evolution}
\frac{d}{dt}\mathscr{E}(v)=-\int_{\bD}\mathscr{M}(v)^2=-\int_{\bD}|v_t|^2\leq 0,
\end{equation}
which in combination with the integral Minkowski inequality and the smooth convergence of $v(\cdot,t)$ to $0$ as $t\rightarrow-\infty$ implies that
\begin{equation}\label{eq_v_L2}
\|v(\cdot,t)\|_{L^2(\mathbb{D})}\leq\int_{-\infty}^t\|v_\tau\|_{L^2(\mathbb{D})}~d\tau=-\mathscr{E}(v(\cdot,t)).
\end{equation}
Note in particular that $\mathscr{E}(v(\cdot,t))\leq 0$. Now \cite[Theorems 2 and Theorem 4]{Feehan-Maridakis} give us constants $Z>0$ and $\sigma>0$ such that if $x\in\mathscr{U}$, then
\begin{equation}\label{eq_simon}
\|x-x_\infty\|_{\mathscr{X}}<\sigma\quad \implies \quad \|\mathscr{M}(x)\|_{\tilde{\mathscr{X}}}\geq|\mathscr{E}(x)-\mathscr{E}(x_\infty)|^\frac{1}{2}.
\end{equation}
Since as set up above, $x_\infty$ is the identically vanishing backward smooth limit of $v(\cdot,t)$ with $\mathscr{E}(x_\infty)=0$, we deduce from \eqref{eq_energy_evolution} and \eqref{eq_simon} the existence of some $S<0$ such that for all $\tau\leq S$,
\[
\frac{d}{dt}\mathscr{E}(v(\cdot,\tau))\leq -Z^2|\mathscr{E}(v(\cdot,\tau))|=-Z^2\mathscr{E}(v(\cdot,\tau)),
\]
or equivalently that $\frac{d}{dt}|\mathscr{E}(v(\cdot,\tau))|\geq -Z^2|\mathscr{E}(v(\cdot,\tau))|$. Integrating backwards from time $S$ therefore yields
\begin{equation}\label{eq_energy_decay}
|\mathscr{E}(\cdot,t)|\leq e^{Z^2(t+T)}|\mathscr{E}(\cdot,S)|,
\end{equation}
and combining with \eqref{eq_v_L2} yields the estimate
\[
\|v(\cdot,t)\|_{L^2(\mathbb{D})}\leq C e^{Z^2 t},
\]
for some constant $C>0$. Finally, using this along the smooth convergence of $v(\cdot,t)$ to $0$ as $t\rightarrow -\infty$ yields by interpolation of Sobolev norms (see for instance \cite[Theorem 5.2]{Fournier-Adams}) the existence of some constants $a,b>0$ such that for all $t\leq S$ we have $\|v(\cdot,t)\|_{C^2(\old)}\leq a e^{bt}$.
\end{proof}

\begin{proof}[Proof of Proposition \ref{prop_symmetry}]
Let $T<0$ and $S<0$ be as in Lemmas \ref{lem_integral_estimate} and \ref{lem_decay_estimate}, respectively. Then for all $t\leq\min\{T,S\}$ we deduce by combining the growth estimate of $\|\phi\|_{L^\infty(\old)}$ in \eqref{eq_integral_estimate} and the decay estimate of $\|v\|_{C^2(\old)}$ in \eqref{eq_decay_estimate} that $\|\phi(\cdot,t)\|_{L^\infty(\old)}=0$. This implies that $v(\cdot,t)$ is rotationally symmetric for all $t\leq \min\{T,S\}$, and finally the short-time existence and uniqueness theory for \eqref{eq_symmetry_v} implies that $v(\cdot,t)$ is rotationally symmetric for all $t<0$.
\end{proof}

\section{Uniqueness of the rotationally symmetric ancient solution}\label{uniquenesssection}
In this section, we prove that the ancient flow we constructed in Section \ref{construction of arf} is unique within the class of rotationally symmetric solutions with prescribed backward limit the flat disk. Combining this with Proposition \ref{bdddiam} and Proposition \ref{prop_g_symmetry}, we then conclude that our constructed ancient solution is the unique solution to \eqref{2DRF} with positive curvature and bounded diameter, modulo a time-independent diffeomorphism and a time-shift.

\subsection{Sharp asymptotics and spectral analysis.}
We take $g(\psi,\theta,t) = u(\psi,t)g_{0}$ to be the ancient radially symmetric {Ricci flow with boundary} solving \eqref{2DRF} which we constructed in Section \ref{construction of arf}. To study its uniqueness, we must obtain its precise asymptotics near the time $t=-\infty$. We shall use the profile function
$$v(\psi,t) := \frac{1}{\sqrt{u(\psi,t)}} - \frac{1}{\sqrt{u_\fl(\psi)}}$$
to describe the asymptotic convergence.
Notice that
\begin{align}
    &\bigg[\log\bigg(\frac{1}{\sqrt{u(\psi,t)}} - \frac{1}{\sqrt{u_\fl(\psi)}}\bigg) - \lda_0^2 t\bigg]_t\notag\\
    &= \frac{u^{-3/2}(\Delta_0 \log(u^{-1/2}) + 1)}{\frac{1}{\sqrt{u}} - \frac{1}{\sqrt{u_\fl}}} -\lda_0^2\notag\\
    &=\frac{1-\frac{1}{2}\log \Delta_0 u}{u - u\sqrt{\frac{u}{u_\fl}}}-\lda_0^2\notag\\
    &= \frac{R(\psi,t)}{2u^{\frac{1}{2}}(\psi,t)}\cdot \frac{1}{u^{-\frac{1}{2}}(\psi,t) - u_\fl^{-\frac{1}{2}}(\psi)} -\lda_0^2, \label{R/h}
\end{align}
where the last equality is due to the formula $uR = 2-\Delta_0\log u$. To obtain the desired unique exponential asymptotics of the ancient flow near the flat disk, we must have
\begin{equation}\label{ldaconv}
    \frac{R(\psi,t)}{h(\psi,t)} := \frac{R(\psi,t)}{u^{\frac{1}{2}}(\psi,t)}\cdot \frac{1}{u^{-\frac{1}{2}}(\psi,t) - u_\fl^{-\frac{1}{2}}(\psi)} \stackrel{t\ra-\infty}{\longrightarrow} 2\lda_0^2
\end{equation}
for {$\lda_0^2 = k^2\mu_0>k^2$}. This is because, by a linearization of the operator acting on $v$ near $t=-\infty$, (\ref{ldaconv}) should give us
$$v(\psi,t)e^{-\lda_0^2 t} = \parent{\frac{1}{\sqrt{u(\psi,t)}} - \frac{1}{\sqrt{u_\fl(\psi)}}}e^{-\lda_0^2 t} \stackrel{t\ra-\infty}{\longrightarrow} Af_0(\psi)$$
where $f_0(\psi)$ is the positive eigenfunction to the linearized operator discussed in Section \ref{robinbcspectrum} and $A>0$ some constant. To obtain (\ref{ldaconv}), we need two lemmas:
\begin{enumerate}
    \item $(\log v(t) - \lda_0^2 t)_t >0$ for all $t<0$, and
    \item $(\log v(t) - \lda_0^2 t)_t \leq f(t) \in L^1(-\infty,T) $ for some $T<0$.
\end{enumerate}
The first monotonicity lemma is to ensure that the limit
$$\lim_{t\ra -\infty}v(\psi,t)e^{-\lda_0^2 t} \in [0,\infty)$$
indeed exists; the second is to show that, via integrating $(\log v - \lda_0^2 t)_t$ from $t=-\infty$ to $T$, the presumed asymptotics are in fact sharp, i.e.
$$\lim_{t\ra -\infty}v(\psi,t)e^{-\lda_0^2 t} \in (0,\infty).$$
\begin{lemma}
    For our constructed ancient solution $u(\psi,t)$, we have
    $$\bigg[\log\bigg(\frac{1}{\sqrt{u(\psi,t)}} - \frac{1}{\sqrt{u_\fl(\psi)}}\bigg) - \lda_0^2 t\bigg]_t > 0$$
    for all $t<0$.
\end{lemma}
\begin{proof}
Taking the boundary derivative $\partial_{\nu_0}$ of the expression in the last line, we have
\begin{align}
    &\partial_{\nu_0}\bigg(\frac{R(\psi,t)}{2u^{\frac{1}{2}}(\psi,t)}\cdot \frac{1}{u^{-\frac{1}{2}}(\psi,t) - u_\fl^{-\frac{1}{2}}(\psi)}\bigg)\nonumber\\
    &= \frac{1}{2}\parent{\frac{R_\psi(\frac{\pi}{2},t)}{\sqrt{u}(\frac{\pi}{2},t)} - \frac{R(\frac{\pi}{2},t)u_\psi(\frac{\pi}{2},t)}{2u^{3/2}(\frac{\pi}{2},t)}}\cdot \frac{1}{u^{-\frac{1}{2}}(\hpi,t) - u_\fl^{-\frac{1}{2}}(\hpi)}\nonumber\\
    &= \frac{1}{2}\parent{\frac{kR\sqrt{u}(\frac{\pi}{2},t)}{\sqrt{u}(\frac{\pi}{2},t)} - \frac{kR(\frac{\pi}{2},t)u^{\frac{3}{2}}(\frac{\pi}{2},t)}{u^{3/2}(\frac{\pi}{2},t)}}\cdot \frac{1}{u^{-\frac{1}{2}}(\hpi,t) - u_\fl^{-\frac{1}{2}}(\hpi)} \stackrel{BC}{=} \textbf{0.} \label{heightBC}
\end{align}
This implies that we can work with the maximum principle on the quantity
$$\frac{R(t)}{\sqrt{u(t)}}\cdot (u^{-1/2}(t) - u_\fl^{-1/2})^{-1} =: Rh^{-1}.$$
Drawing analogies to mean curvature flow, the profile function
\begin{equation}\label{eq_h}
h(\psi,t) = 1 - \sqrt{\frac{u(\psi,t)}{u_\fl(\psi)}}>0
\end{equation}
corresponds to the ``graphical height'' of the metric $u(\psi,t)$ from the flat disk metric $u_\fl(\psi)$ which is our backward limit. Note that a direct computation plus the inequality (\ref{gradu}) yields $h_\psi\geq 0$:
\begin{align*}
    \partial_\psi h &= \parent{1-\sqrt{u(\psi,t)}\cdot k(1+\cos\psi)}_\psi\\
    &= -\frac{1}{2}\frac{u_\psi(\psi,t)}{\sqrt{u(\psi,t)}}\cdot k(1+\cos\psi) + k\sin\psi\sqrt{u(\psi,t)}\\
    &\stackrel{\eqref{gradu}}{\geq} \parent{-k\tan\parent{\frac{\psi}{2}}(1+\cos\psi) +k\sin\psi}\sqrt{u(\psi,t)} = 0.
\end{align*}
To apply maximum principle on $R/h$, we first compute the Laplacian of $R/h$ with respect to the time-dependent Laplacian $\Delta$ \eqref{time dependent lg}.
Via direct computation, one gets
\begin{align*}
    \Delta \sqrt{\frac{u}{u_\fl}} &= \frac{1}{u}\cdot \sqrt{\frac{u}{u_\fl}}\parent{\frac{1}{2}\big(\Delta_0 \log u - 2\big) + \frac{1}{4}\norm{\nabla_0 \log \frac{u}{u_\fl}}_0^2}\\
    &= \sqrt{\frac{u}{u_\fl}}\parent{\frac{1}{2}\big(\Delta \log u - \frac{2}{u}\big) + \frac{1}{4u}\norm{\nabla_0 \log \frac{u}{u_\fl}}_0^2}\\
    &= \sqrt{\frac{u}{u_\fl}}\parent{-\frac{R}{2} + \norm{\nabla \log\sqrt{\frac{u}{u_\fl}}}^2}.
\end{align*}
Since
\begin{align*}
    \partial_{\nu_{0}}\frac{\sqrt{u}}{\sqrt{u_\fl}} &= \frac{1}{2}\cdot\frac{u_{\nu_0}}{\sqrt{uu_\fl}} - \frac{1}{2}\cdot\frac{\sqrt{u}\cdot (u_\fl)_{\nu_0}}{u_\fl^{3/2}}\\
    &= \frac{1}{2}\parent{\frac{2ku}{\sqrt{u_\fl}} - 2k\sqrt{u}}\\
    &= k\sqrt{u}\parent{\frac{\sqrt{u}}{\sqrt{u_\fl}} - 1} <0,
\end{align*}
it follows that the spatial maximum of $\sqrt{u(t)/u_\fl}$ (therefore the spatial minimum of $h$) cannot be achieved on the boundary.
Also note that
$$\partial_t \sqrt{\frac{u(t)}{u_\fl}} =\sqrt{\frac{u}{u_\fl}}\cdot \frac{1}{2}\parent{\Delta \log u - \frac{2}{u}}.$$
Then $h(\psi,t)$ satisfies the below equation:
\begin{equation}\label{h evolution}
    \begin{cases}
    (\partial_t - \Delta)h = \frac{|\nabla h|_t^2}{1-h}\\
    h_{\nu_t} = \frac{1}{\sqrt{u(t)}}h_{\nu_0} = kh >0.
\end{cases}
\end{equation}
The scalar curvature $R = R(\psi,t)$ satisfies the evolution equation
$$(\partial_t - \Delta)R = R^2.$$
Combining these two, we obtain
\begin{align*}
    (\partial_t - \Delta)(Rh^{-1}) = \frac{2}{h}\left\langle \nabla h,\nabla \parent{\frac{R}{h}}\right\rangle + \parent{h\parent{\frac{R}{h}}^2 - {\frac{R}{h^2(1-h)}\norm{\nabla h}^2}}.
\end{align*}
To apply the maximum principle on $Rh^{-1}(\psi,t)$, we need to prove the following: 
$$R(t) \geq \frac{\norm{\nabla h}^2(t)}{h(t)(1-h(t))} \text{ for all }t<0.$$
To do this, we use a maximum principle on the quantity
\begin{equation}\label{eq_E}
E(t) := R(t) - \frac{\norm{\nabla h}^2}{h(1-h)}
\end{equation}
for old solutions $g_i(t) = u_i(t)g_0$. Indeed, by Lemma \ref{evolution of E} and computing the boundary condition, we have 
$$\begin{cases}
    (\partial_t - \Delta)E \geq E^2,\\
    \partial_{\nu_{g(t)}}E = k^3\frac{h}{1-h}\parent{\frac{1}{1-h} +2} >0.
\end{cases}$$
The above equation implies that $E$ cannot attain spatial minimum on the boundary. In the following, we verify that the quantity $E$ is nonnegative on the initial data $g_i(\tau_i) =  U(\cdot,\tau_i)g_0 =\frac{d\psi^2+ \sin^2\psi \;d\theta^2}{w_\infty^2(1+e^{k^2\mu_0\tau_i}f_0(\psi))^2}$ of the old solutions $g_i$. The height function $h_i$ of $g_i(\tau_i)$ is equal to
$$h_i = 1-\frac{\sqrt{U(\tau_i)}}{\sqrt{u_\fl}} = \frac{e^{\mu_0k^2\tau_i}}{1+e^{\mu_0k^2\tau_i}f_0} = \frac{\eps(\tau_i)}{1+\eps(\tau_i)f_0}.$$
Recall that the scalar curvature of the metric $U(\tau_i)$ is given by
$$R_i(r) =  2k^2\eps(t)\parent{\mu_0f_0(r) + \eps(t)(\mu_0f_0^2(r)-(\partial_rf_0)^2)}.$$
Then we have
$$E_i := E(U(\cdot,\tau_i)) = \eps(\tau_i)\parent{2k^2\mu_0f_0 - k^2(\partial_rf_0)^2/f_0} + O(\eps(\tau_i))^2$$
which is positive for sufficiently small $\tau_i <<-1$: {Lemma \ref{f_0'(r)}} guarantees the positivity of the first term. Therefore, for all old solutions, $E(t)$ remains nonnegative throughout the evolution. Then from the maximum principle it follows that
$$\partial_t \min_\psi \parent{\frac{R_i(t)}{h_i(t)}} \geq 0,$$
implying
$$\inf_{t\in (-\infty,0),\psi\in [0,\pi/2)} \frac{R(\cdot,t)}{h(\cdot,t)} = \lim_{t\ra -\infty}\frac{R(\cdot,t)}{h(\cdot,t)} \swarrow \frac{R_i(\cdot,T_i)}{h_i(\cdot,T_i)} \geq 2\lda_i^2.$$
Below we compute $R(U(t))/h(t)$.
\begin{align*}
    \frac{R(U(r,t))}{h(U(r,t))} &= \frac{R(U(r,t))}{1-\sqrt{\frac{u_\fl(r)}{U(r,t)}}}\\
    &=\frac{2k^2\eps(t)\parent{\mu_0f_0(r) + \eps(t)(\mu_0f_0^2(r)-(\partial_rf_0)^2)}}{1- \frac{1}{1+\eps(t)f_0(r)}}\\
    &= 2k^2(1+\eps(t)f_0(r))\frac{\mu_0f_0(r) + \eps(t)(\mu_0f_0^2(r)-(\partial_rf_0)^2)}{f_0(r)}\\
    &= 2k^2\mu_0 + 2k^2\eps(t)\frac{\mu_0f_0^2 + (1+\eps(t))(\mu_0f_0^2(r) - (\partial_rf_0)^2)}{f_0(r)}\\
    &\geq 2k^2\mu_0 + 2k^2\eps(t)\mu_0f_0(r).
\end{align*}
As the constructed old solutions satisfy $$R_i/h_i \geq 2\lda_i^2 = \frac{R_i(0)}{h_i(0)},\quad \lim_{i\to\infty} \lda_i^2 = k^2\mu_0 = \lda_0^2,$$ taking a limit $i\ra \infty$ completes the proof.
\end{proof}
The above lemma therefore yields the monotonicity of the quantity
$$\parent{\frac{1}{\sqrt{u(t)}} - \frac{1}{\sqrt{u_\fl}}}e^{-\lda_0^2 t}$$
which implies the existence of the limit
$$\lim_{t\ra-\infty} v(\psi,t)e^{-\lda_0^2 t} \in [0,\infty).$$
Now we show that the above limit is indeed positive.
\begin{lemma}\label{R uperbound estimates}
    There exist constants $T>-\infty$ and $C>0$ such that, for sufficiently old solutions (hence so for our constructed ancient solution), we have
    $$R(t) \leq Ce^{2k^2t},\; \forall t<T.$$
\end{lemma}
\begin{proof} 
    Under the Ricci flow, the function $|\nabla R|$ evolves by the equation
    \begin{align*}
        (\partial_t-\Delta)|\nabla R| &= 2R\norm{\nabla R} - \frac{\norm{\nabla^2R}^2 - \norm{\nabla\norm{\nabla R}}^2}{|\nabla R|}\\
        &\stackrel{\text{Kato}}{\leq} 2R\norm{\nabla R}.
    \end{align*}
    Define the function
    \begin{equation}\label{Stestf}
        S(\psi,t) := \sin\psi\cdot \frac{h_\psi}{1-h} - h \geq -h .
    \end{equation}
    Writing in terms of $u(\psi,t)$ and $u_\fl$,
    $$S(\psi,t) = \frac{1}{2}\sin\psi\parent{\partial_\psi\log u_\fl - \partial_\psi \log u(t)} -  \parent{1 - \sqrt{\frac{u(t)}{u_\fl}}}.$$
    As $u_\fl$ and $u(t)$ satisfy the same boundary condition,
    \begin{align*}
    S\parent{\hpi,t} &= \parent{\frac{1}{2}\parent{\frac{u_{\fl,\psi}}{u_\fl} - \frac{u_\psi}{u}} - h}\parent{\hpi,t} 
    \\
    &= \parent{k\sqrt{u_\fl(\pi/2)}- 1}h\parent{\hpi,t} 
    \\
    &= 0.
    \end{align*}
    Via direct computation, we find that $S$ satisfies the evolution equation
    $$(\partial_t - \Delta)S = R(S-1) + Rh - \frac{\norm{\nabla h}^2}{1-h}.$$
   For
    $$\phi_\eps := \norm{\nabla R} - 2k\sqrt{\mu_0}R + 2S - \eps e^{t-T},$$
    we have
    \begin{align*}
        (\partial_t-\Delta)\phi_\eps &\leq 2R\norm{\nabla R} - 2k\sqrt{\mu_0}R^2 + 2RS - 2R - \frac{2\norm{\nabla h}^2}{1-h} - \eps e^{t-T}\\
        &\leq 2R\norm{\nabla R} - 2k\sqrt{\mu_0}R^2 + 2RS - 2R -   \eps e^{t-T}
    \end{align*}
    where $\eps>0$ is a small constant.
    By our construction of the ancient solution, each old solution $u_i$ satisfies $\phi_\eps <0$ at the initial timeslice: for the metric
    \begin{align*}
        G(t) := U(\psi,t)(d\psi^2+ \sin^2\psi \;d\theta^2) =  \frac{d\psi^2+ \sin^2\psi \;d\theta^2}{w_\infty^2(1+e^{k^2\mu_0 t}f_0(\psi))^2}
    \end{align*}
    which we choose our initial data for old solutions from, we have, by {Lemma \ref{f_0'(r)}},
    \begin{align*}
        \lim_{t \to -\infty} \frac{|\nabla R|}{R}(G(t)) = k(1 + \cos \psi)\frac{\partial_\psi f_0}{f_0} \leq k\sqrt{\mu_0}.
    \end{align*}
    
   Let $u = u_i$ denote any old solution which survives for a sufficiently long time. By the boundary condition $\norm{\nabla R} = kR$ plus the fact that $u(\psi,t) \leq 1/k^2$, the function $\phi_\eps = \phi_\eps(u_i)$ is always negative on the boundary. So any nonnegative parabolic maximum must occur within the interior.

    Hence at the point $(\psi_0,t)$ where $\phi$ achieves parabolic maximum
    $$\sup_{\tau\leq t,\psi\in [0,\hpi]} \phi(\psi,\tau) = \phi(\psi_0,t) =0,$$
    we have
    \begin{align*}
        0\leq (\partial_t-\Delta)\phi_\eps
        &\leq 2R\norm{\nabla R} - 2k\sqrt{\mu_0}R^2 + 2RS - 2R - \eps e^{t-T}\\
        &= 2R(k\sqrt{\mu_0}R -S+\eps e^{t-T}-1) - \eps e^{t-T}\\
        &\leq 2R(k\sqrt{\mu_0}R +h  -1) + (2R-1)\eps e^{t-T} <0 
    \end{align*}
    for sufficiently small $t\leq T$, leading to a contradiction. Due to the uniform curvature estimate {Lemma \ref{C0curvbound}}, this $T$ can be chosen independent of the index $i$, given $i \geq i_0$ sufficiently large. Taking a limit $i\ra \infty$ and $\eps \to 0$, we find that, for our old solution $u_\infty(t)$ we have
    $$|\nabla R_\infty| \leq 2k\sqrt{\mu_0}R -2S \leq (2k\sqrt{\mu_0} + 2k^2\mu_0)R \text{ for all } t\leq T$$
    where the last inequality follows from the previous lemma $R_\infty/h_\infty \geq 2k^2\mu_0$ and \eqref{Stestf}. Integrating $\norm{\nabla R_\infty}/R$ from $\psi=0$ to $\hpi$, we have
    $$\frac{R(\hpi,t)}{R(0,t)} \leq C(k,\mu_0,T)$$
    for $t\leq T$. The bounds in {Lemma \ref{scalandconformal}} then yield
    $$R(\psi,t) \leq R\left(\hpi,t\right) \leq C(k,\mu_0,T)R(0,t) \leq C'(k,\mu_0,T)e^{2k^2 t}.$$
\end{proof}

\begin{lemma}
    There exists $a,C>0$ and $T<0$ such that, for all $t<T$, we have
    $$\bigg[\log\bigg(\frac{1}{\sqrt{u(\psi,t)}} - \frac{1}{\sqrt{u_\fl(\psi)}}\bigg) - \lda_0^2 t\bigg]_t \leq Ce^{k^2t} \in L^1(-\infty,T].$$
\end{lemma}
\begin{proof}
    Recall the evolution equation
    \begin{align*}
    (\partial_t - \Delta)(Rh^{-1}) &= \frac{2}{h}\left\langle \nabla h,\nabla \parent{\frac{R}{h}}\right\rangle + \parent{h\parent{\frac{R}{h}}^2 - \frac{R}{h^2(1-h)}\norm{\nabla h}^2}\\
    &\leq 2\left\langle \frac{\nabla h}{h},\nabla \parent{\frac{R}{h}}\right\rangle + R\cdot \frac{R}{h}\\
    &\leq 2\left\langle \frac{\nabla h}{h},\nabla \parent{\frac{R}{h}}\right\rangle + Ce^{2k^2 t}\cdot \frac{R}{h}
\end{align*}
where the last inequality follows from the previous lemma, for sufficiently small $T_i\leq t<T$. 
Then for any old solution $u_i$, we have
$$\max_{t,\psi} \frac{R_i}{h_i} \leq \max_{\psi} (Ce^{2k^2t}+1)\frac{R_i(T_i)}{h_i(T_i)}$$
due to the boundary condition
$$\partial_\nu \frac{R}{h} \equiv 0$$
and an ODE comparison.
Then using the bound $R_i/h_i \geq 2\lda_i^2$, we obtain the inequality for the constructed ancient solution
$$ \frac{R_\infty}{h_\infty} \leq 2k^2\mu_0 + C(k,\mu_0)e^{2k^2t} \text{ for }t\leq T $$
by sending $2\lda_i^2 \ra 2k^2\mu_0$ ($i\to \infty$). Recalling the identity (\ref{R/h}), we have
\begin{align*}
    \bigg[\log\bigg(\frac{1}{\sqrt{u(\psi,t)}} - \frac{1}{\sqrt{u_\fl(\psi)}}\bigg) - \lda_0^2 t\bigg]_t &= \frac{R_\infty}{2h_\infty} - k^2\mu_0\\
    &\leq Ce^{2k^2 t}.
\end{align*}
\end{proof}
Following from the discussion at the beginning of this section, we obtain:
\begin{proposition}[Sharp asymptotics]
    For our constructed ancient solution to \eqref{2DRF} $u_\infty(\psi,t)$, the limit function
    $$A_\infty(\psi) := \lim_{t\to -\infty} \parent{\frac{1}{\sqrt{u_\infty(\psi,t)}} - \frac{1}{\sqrt{u_\fl(\psi,t)}}}e^{-k^2\mu_0t} \in (0,\infty)$$
    exists.
\end{proposition}

\subsection{Uniqueness} Now let $({g}(t) := {u}(\psi,t)g_0)_{t<0}$ denote \textit{any} rotationally symmetric {Ricci flow with boundary}, with a flat disk backward limit $(\old,g_{\fl})$. Recalling its properties in Section \ref{subsec_anyancient}, we first prove that ${g}(t)$ must satisfy the same asymptotics as our constructed solution $g_\infty(t)$ near $t=-\infty$.
\begin{proposition}\label{anyasymp}
Define $A := A_\infty(0) >0$. We have
    $$\lim_{t\to -\infty} e^{-k^2\mu_0 t}\parent{\frac{1}{\sqrt{{u}(\psi,t)}} - \frac{1}{\sqrt{u_\fl(\psi)}}} = Af_0(\psi).$$
\end{proposition}
\begin{proof}
    Denote $$v(\psi,t) := \frac{1}{\sqrt{{u}(\psi,t)}} - \frac{1}{\sqrt{u_\fl(\psi)}}$$
    and its rescaled version
    $$v^\tau(\psi,t) := e^{-k^2\mu_0\tau}v(t+\tau) =   e^{-k^2\mu_0\tau}\parent{\frac{1}{\sqrt{{u}(\psi,t)}} - \frac{1}{\sqrt{u_\fl(\psi)}}}$$
    where $\tau<0$. $v^\tau$ is defined for $t\in (-\infty,-\tau)$. $v^\tau$ satisfies the equation
\begin{align}
    \partial_t \log v^\tau = \frac{v^\tau_t}{v^\tau} &= u(t+\tau)^{-3/2}\cdot \frac{\Delta_0\log (v(t+\tau) + w_\infty) - \Delta_0 \log w_\infty}{v(t+\tau)}\label{vtau}\\
    &= u(t+\tau)^{-3/2}\cdot \frac{\Delta_0\log (e^{k\mu_0^2\tau}v^\tau(t) + w_\infty) - \Delta_0 \log w_\infty}{e^{k\mu_0^2\tau}v^\tau(t)}\notag,
\end{align}
with the boundary condition
$$\partial_\nu v^\tau = 0.$$
Since $u(t)$ must vanish at $t=0$, by the comparison principle Lemma \ref{maxtime}, $u(t)$ must intersect $u_\infty(t)$ at all times. Especially, the lower bound of one cannot exceed the upper bound of another. The same holds for the functions $h(t)$ and $h_\infty(t)$. Also note that the nonnegativity of the gradients $u_\psi$ and $R_\psi$ gives $h_\psi\geq 0$ just as in the case of the constructed (old) solutions. Then {Lemma \ref{anysolnbound}} implies
$$\limsup_{t\to-\infty}e^{-k^2t}\max_{\psi}{h}(\psi,t) \leq C <\infty.$$
Then we can use Alaoglu's theorem to take a uniform limit on \eqref{vtau} by choosing a sequence $\tau_j \to -\infty$. The resulting weak limit function $\Phi := \lim_{\tau_j\to-\infty} v^{\tau_j}$ must solve the linearized equation of \eqref{vtau}:
\begin{equation}
\begin{cases}
   \Phi_t =  w_\infty^3\Delta_0 \parent{\frac{\Phi}{w_\infty}},\\
    \partial_{\nu_0}\Phi = 0.
\end{cases}
\end{equation}
From {Proposition \ref{rulingoutnull}}, we conclude that
$$\Phi(\psi,t) = Ae^{k^2\mu_0t}f_0(\psi)$$
where $A := A_\infty(0).$
\end{proof}

Now we are ready to prove the uniqueness of ancient solution to \eqref{2DRF} with a flat disk backward limit.
\begin{proposition}[Uniqueness]\label{uniqueness of 1.1}
    Modulo a translation in time and a time-independent diffeomorphism, there is only one positively curved ancient solution to the {Ricci flow with boundary} equation \eqref{2DRF} satisfying conditions of constant boundary geodesic curvature $k>0$ and uniformly bounded diameter.
\end{proposition}
\begin{proof}
  By  {Proposition \ref{flatdisklimit}} and Proposition \ref{prop_g_symmetry}, we only need to show the uniqueness of rotationally symmetric ancient solutions to \eqref{2DRF} with the prescribed backward limit $(\old,g_{\fl} = u_\fl g_0)$. Denote by $u_\infty(t)$ the constructed ancient solution and $u(t)$ another rotationally symmetric ancient solution to the flow equation \eqref{2DRF} which converges backward to $g_{\fl}$. Modulo a time shift, we assume that both flows vanish at time $t=0$. Denote, for $\tau<0$:
    $$v_\infty(t) := \frac{1}{\sqrt{{u_\infty}(\psi,t)}} - \frac{1}{\sqrt{u_\fl(\psi)}}, \quad  v_\tau(t):= \parent{\frac{1}{\sqrt{{u}(\psi,t+\tau)}} - \frac{1}{\sqrt{u_\fl(\psi)}}}.$$
    By {Proposition \ref{anyasymp}},
    $$e^{-k^2\mu_0 t}v_\tau(\psi,t) \to Ae^{k^2\mu_0\tau}f_0(\psi) < Af_0(\psi) \text{ as } t\to -\infty $$
    uniformly in $\psi.$ Therefore the time-shifted function $v_\tau(t)$ must stay above $v_\infty(t)$ for sufficiently small $t<0$ and so should $u(t+\tau)$ stay below $u_\infty(t)$. The comparison principle {Lemma \ref{maxtime}} yields that $u(t+\tau)$ must stay below $u_\infty(t)$ for all $t<0$. Taking $\tau \to 0$, we find that in fact $$u(t)\leq u_\infty(t) \text{ for all } t<0.$$ But as both $u$ and $u_\infty$ should vanish at time $t=0$, they must intersect for all $t<0$, yielding $u \equiv u_\infty$ for all $t$ again by {Lemma \ref{maxtime}}.
\end{proof}

\appendix
\begin{appendices}

\section{Evolution equation of \texorpdfstring{$E(\psi,t)$}{E(psi,t)}}\label{EvolE}
Below we derive the following differential inequality for the quantity $E$ defined in \eqref{eq_E}, which we used in the proof of the uniqueness of rotationally symmetric ancient solutions in Section \ref{uniquenesssection}.
\begin{lemma}\label{evolution of E}
Let $g(\psi,\theta,t) = u(\psi,t)g_{0}$, $0\leq t< T$ be a radially symmetric {Ricci flow with boundary} on $\old$ solving \eqref{2DRF}. The quantity $$E(\psi,t) := R(\psi,t) - \frac{|\nabla h|^2}{h(1-h)}$$
defined in \eqref{eq_E} satisfies
$$\partial_tE - \Delta_{g(t)} E \geq E^2.$$
Above, $R$ denotes the scalar curvature, $h(\psi,t) = 1 - \sqrt{\frac{u(\psi,t)}{u_\fl(\psi)}}$ as in \eqref{eq_h}, and $u_{\fl}$ is the conformal factor associated with the flat disk as in \eqref{eq_flat}.
\end{lemma}
\begin{proof}

For brevity, we omit the time-dependence of the norms and derivatives in terms of the flow metric $g(t)$, unless specified otherwise. We first consider the evolution equation satisfied by $|\nabla h|^2$. By the Ricci flow equation, we have
$$\partial_t \norm{\nabla h}^2 = R\norm{\nabla h}^2 + \langle \nabla h,\nabla\partial_th \rangle$$
and by the Bochner formula
\begin{align*}
    \Delta |\nabla h|^2 = 2\norm{\nabla^2 h}^2 + 2\langle \nabla h,\nabla \Delta h\rangle + R|\nabla h|^2.
\end{align*}
Then we compute the evolution equation of $\norm{\nabla h}^2$:
\begin{align*}
    (\partial_t -\Delta)\norm{\nabla h}^2 &= -2|\nabla^2 h|^2 + 2\langle \nabla h, \nabla(\partial_t-\Delta)h\rangle\\
    &=-2|\nabla^2 h|^2 + 2\langle \nabla h, \nabla\frac{\norm{\nabla h}^2}{1-h}\rangle\\
    &= -2|\nabla^2 h|^2 + \frac{4}{1-h}\nabla^2 h(\nabla h,\nabla h)
\end{align*}
where the last equality follows from the identity $\nabla\norm{\nabla h}^2 = 2\nabla^2h (\nabla h,\cdot)$ and the penultimate equality follows from the evolution equation satisfied by $h$.

Since $h$ satisfies the equation $(\partial_t - \Delta)h = \frac{\norm{\nabla h}^2}{1-h}$, the function $f := \log h$ satisfies the evolution equation
$$(\partial_t - \Delta)f = \frac{\norm{\nabla h}^2}{h(1-h)} = \frac{h}{1-h}\norm{\nabla f}^2 = R-E.$$
and
\begin{equation}
    \Delta f = \frac{R(1-h)}{2h} - \frac{\norm{\nabla h}^2}{h(1-h)}.\label{B1}
\end{equation}
Again by the Bochner formula, $\norm{\nabla f}^2$ satisfies the evolution equation
\begin{align*}
    (\partial_t- \Delta)\norm{\nabla f}^2 &= 2\langle \nabla f,\nabla(\partial_t- \Delta)f\rangle - 2\norm{\nabla^2 f}^2\\
    &= \frac{2}{1-h}\langle \nabla f, \nabla \norm{\nabla f}^2\rangle + \frac{2h}{(1-h)^2}\norm{\nabla f}^2 - 2\norm{\nabla^2 f}^2.
\end{align*}
Via a direct computation, we have
\begin{align*}
     (\partial_t- \Delta)E &=  (\partial_t- \Delta)\parent{\frac{h}{1-h}\norm{\nabla f}^2}\\
     &= -\frac{2h}{1-h}\norm{\nabla^2 f}^2 + \frac{h^2}{(1-h)^3}\norm{\nabla f}^4.
\end{align*}
Combined with the evolution equation of scalar curvature $R(t)$ under the Ricci flow, we have
\begin{align}
     (\partial_t- \Delta)E -E^2&= \frac{2h}{1-h}\norm{\nabla^2 f}^2 +\frac{2R\norm{\nabla h}^2}{h(1-h)} - \parent{1+\frac{1}{1-h}}\parent{\frac{\norm{\nabla h}^2}{h(1-h)}}^2\label{B2}
\end{align}
Now we  observe that
\begin{align*}
   \norm{\nabla^2 f + \frac{|\nabla h|^2}{2h(1-h)}g}^2 &= \norm{\nabla^2 f}^2 + \frac{|\nabla h|^2}{h(1-h)}\Delta f + \frac{1}{2}\parent{\frac{|\nabla h|^2}{h(1-h)}}^2\\
   &= \norm{\nabla^2 f}^2 + \frac{R\norm{\nabla h}^2}{2h^2}- \frac{1}{2}\parent{\frac{|\nabla h|^2}{h(1-h)}}^2
\end{align*}
where we have used the equation (\ref{B1}) in the last equality. Then we have
\begin{align*}
   &\frac{2h}{1-h}\norm{\nabla^2 f + \frac{|\nabla h|^2}{2h(1-h)}g}^2 + \frac{R|\nabla h|^2}{h(1-h)} \\
   &= \frac{2h}{1-h}\norm{\nabla^2 f}^2 + \frac{2R|\nabla h|^2}{h(1-h)} + \frac{h-2}{1-h}\parent{\frac{|\nabla h|^2}{h(1-h)}}^2\\
   &= \frac{2h}{1-h}\norm{\nabla^2 f}^2 + \frac{2R|\nabla h|^2}{h(1-h)} -\parent{1+\frac{1}{1-h}}\parent{\frac{|\nabla h|^2}{h(1-h)}}^2.
\end{align*}
Substituting the above equality into the evolution equation (\ref{B2}), we have
\begin{equation*}
    (\partial_t- \Delta)E -E^2= \frac{2h}{1-h}\norm{\nabla^2 f + \frac{|\nabla h|^2}{2h(1-h)}g}^2 + \frac{R|\nabla h|^2}{h(1-h)} \geq 0.
\end{equation*}

\end{proof}

\end{appendices}

\printbibliography
\end{document}